\documentclass [11pt,reqno]{amsart}
\usepackage[dvips,bookmarks=false]{hyperref}
\usepackage {amsmath,amssymb,verbatim}

\usepackage{geometry}
\newcommand{\C}{\mathbf{C}}

\newcommand{\Q}{\mathbf{Q}}
\newcommand{\R}{\mathbf{R}}

\newcommand{\F}{\mathbf{F}}
\newcommand{\Z}{\mathbf{Z}}
\newcommand{\N}{\mathbf{N}}
\renewcommand{\P}{\mathbf{P}}

\newcommand{\fX}{\mathfrak{X}}

\newcommand{\tF}{\tilde{F}}

\newcommand{\cL}{\mathcal{L}}

\newcommand{\cO}{\mathcal{O}}
\newcommand{\cT}{\mathcal{T}}
\newcommand{\cV}{\mathcal{V}}

\newcommand{\hcV}{\hat{\mathcal{V}}}
\newcommand{\cVdiv}{{\mathcal{V}_{\mathrm{div}}}}
\newcommand{\hcVdiv}{{\hat{\mathcal{V}}_{\mathrm{div}}}}

\newcommand{\Ltwo}{\mathrm{L}^2}

\renewcommand{\a}{\alpha}
\renewcommand{\b}{\beta}
\newcommand{\g}{\gamma}
\renewcommand{\d}{\delta}

\newcommand{\NS}{\mathrm{NS}}

\newcommand{\e}{\varepsilon}

\newcommand{\eg}{{\rm e.g.\ }} 
\newcommand{\ie}{{\rm i.e.\ }}

\newcommand{\la}{\lambda}

\newcommand{\tnu}{\tilde\nu}

\newcommand{\vv}{\vec{v}}

\renewcommand{\=}{:=}

\DeclareMathOperator{\id}{id}

\DeclareMathOperator{\lcm}{lcm}

\DeclareMathOperator{\ord}{ord}
\DeclareMathOperator{\Pic}{Pic}
\DeclareMathOperator{\hodge}{NS}

\numberwithin{equation}{section}       % Number formulas within sections

\newtheorem{prop}{Proposition}[section]
\newtheorem{thm}[prop]{Theorem} 
\newtheorem{defi}[prop]{Definition}
\newtheorem{lem}[prop]{Lemma}
\newtheorem{cor}[prop]{Corollary}
\newtheorem{prop-def}[prop]{Proposition-Definition}

\newtheorem*{thmA}{Theorem A} 
\newtheorem*{thmB}{Theorem B} 
\newtheorem*{thmC}{Theorem C} 
\newtheorem*{thmD}{Theorem D}

\theoremstyle{remark}
\newtheorem{exam}[prop]{Example}
\newtheorem{rmk}[prop]{Remark}

\newtheorem*{ackn}{Acknowledgment} 

\title{Dynamical compactifications of $\C^2$}

\date{\today}

\author{Charles Favre \and Mattias Jonsson}
\address{CNRS-Universit{\'e} Paris 7\\
  Institut de Math{\'e}matiques\\
  F-75251 Paris Cedex 05\\
  France}
\email{favre@math.jussieu.fr}
\address{Dept of Mathematics\\
  University of Michigan\\
  Ann Arbor, MI 48109-1043\\
  USA}
\email{mattiasj@umich.edu}
\thanks{Second author supported by the NSF
  and the Swedish Research Council.}

\begin{document}

\begin{abstract}
  We find good dynamical compactifications for arbitrary
  polynomial mappings of $\C^2$ and use them to show that
  the degree growth sequence satisfies
  a linear integral recursion formula. 
  For maps of low topological degree we 
  prove that the Green function is well behaved.
  For maps of maximum topological degree, we give 
  normal forms.
\end{abstract}

\maketitle

%%%%%%%%%%%%%%%%%%%%%%%%%%%%%%%%%%%%%%%%%%%%%%%%%%%%%%%%%%%%%%%%%%
%\setcounter{tocdepth}{1}
%\tableofcontents
%
%
%%%%%%%%%%%%%%%%%%%%%%%%%%%%%%%%%%%%%%%%%%%%%%%%%%%%%%%%%%%%%%%%%%%
%
%

\section*{Introduction}
The theory of iteration of rational maps on complex projective
varieties has recently seen the introduction of new analytic
techniques for constructing invariant currents and measures of
dynamical interest,
through the work of
Bedford-Diller~\cite{BD}, de Th{\'e}lin-Vigny~\cite{DTV}, 
Diller-Dujardin-Guedj~\cite{DDG1,DDG2,DDG3}, Dinh-Sibony~\cite{DSallure,DSgeom,DSsuper},
Dujardin~\cite{henon-like,laminar},
Guedj~\cite{guedjpoly,guedjannals},
and others.
These constructions, however,
often require a good birational model in which the 
dynamical indeterminacy
set has a relatively small size, so that the action on cohomology of
the rational map is compatible with iteration.

Diller-Favre~\cite{DF} proved the existence of such models for
birational surface maps using the decomposition into
blow-ups and blow-downs. There are no other general results, 
for two reasons. First, it is a delicate task to control
a rational map near its indeterminacy set. 
Second, the indeterminacy
set of a non-invertible map tends to grow very 
rapidly under iteration.

In this paper we prove the existence of good birational models for 
an important class of 
rational surface maps, namely polynomial maps. 
\begin{thmA}
  Let $F:\C^2\to\C^2$ be any polynomial mapping. 
  Then there exists a projective compactification
  $X\supset\C^2$ with at worst quotient singularities
  and an integer $n\ge1$ such that the lift
  $\tF:X\dashrightarrow X$ satisfies
  $\tF^{(j+n)*}=\tF^{j*}\tF^{n*}=(\tF^{*})^j\tF^{n*}$
  on the Picard group $\Pic(X)$ for all $j\ge1$.
\end{thmA}
Fix an algebraic embedding $\C^2\subset\P^2$.
Define $\deg(F)$ by the relation $F^*\cL = \deg(F)\cL$ 
for $\cL$ a generator of $\Pic(\P^2)$. 
Using Theorem~A we prove
\begin{thmB}
  For any  polynomial mapping $F:\C^2\to\C^2$,
  the sequence $(\deg(F^j))_{j\ge0}$ satisfies an
   integral linear recursion formula.
\end{thmB}
This gives a positive answer to the main conjecture
of Bellon-Viallet~\cite[\S5]{bellon-viallet} in our setting.
Example~\ref{ex:recursion} shows that there 
may not be a recursion of order one or two,
even though the asymptotic degree $\la_1$ (see below) is always a quadratic
integer~\cite{eigenval}.
In general, the degree growth of rational maps of 
projective space remains mysterious, despite 
recent works~\cite{BK1,BK2,BK3,BF,propp,Nguyen}.
Hasselblatt and Propp~\cite{propp} give
examples of rational maps of $\C^2$ for
which the degree growth does not satisfy any linear 
recursion formula.
From this perspective, Theorem~B is quite remarkable.

\smallskip
In many cases---and always after replacing $F$ by $F^2$---$X$ 
is smooth and can be obtained from $\P^2$ after 
finitely many blowups at infinity. 
In these cases, Theorem~B follows immediately from Theorem~A.

Note that the compactifications that we consider here are
different in nature from the ones studied by Hubbard, Papadopol 
and Veselov~\cite{HPV}.

\smallskip
The conclusion in Theorem~A is slightly weaker than the 
condition $(\tF^j)^*=(\tF^*)^j$ on $\Pic(X)$ for all $j$; 
the latter is often referred to as 
\emph{algebraic stability}~\cite{Sib}.
Its importance was first recognized by 
Forn{\ae}ss and Sibony~\cite{FS2}.
Guedj has conjectured~\cite[Remark~3.1]{guedjpoly} 
than any polynomial mapping of $\C^2$ is algebraically
stable on some (smooth) compactification $X$ of $\C^2$. 
While we suspect this may be too much to ask for,
see Remark~\ref{r:guedjcounter}, Theorem~A shows that the 
conjecture holds after replacing $F$ by an iterate. 
In any case, Theorem~A is sufficient for all known applications.

\smallskip
The proof of Theorem~A is based on the 
valuative techniques developed in~\cite{eigenval}. These give
a framework for studying 
the dynamics induced by $F$ on the set of divisors
at infinity in all compactifications of $\C^2$. 
This set of divisors can be identified with a dense
subset of a metrized $\R$-tree $\cV_0$ consisting of all
valuations on $\C[x,y]$, centered at infinity and suitably 
normalized. There are however two difficulties in working directly 
with $\cV_0$. First, a valuation
in $\cV_0$ is a local object, whereas we are interested in 
global properties of $F$. Second, $F$ 
might be not proper, and in this case it does not preserve $\cV_0$. 
To remedy  these problems,
we introduced in~\cite{eigenval} 
a subtree $\cV_1$ of $\cV_0$ consisting of valuations close enough
to $-\deg$, see~\S\ref{sec:tight} for a formal definition. 
This subtree is a fundamental technical tool in our analysis.
In op.\ cit., we showed that valuations in $\cV_1$ still capture
global information, that $F$  induces a continuous map on $\cV_1$, 
and we proved the existence of 
a locally attracting  valuation $\nu_* \in \cV_1$ that 
we called an \emph{eigenvaluation}. Theorem~A is a 
consequence of a detailed study of  the global 
contracting properties of $F$ on $\cV_1$.

Denote by 
$\la_1\=\lim_{n\to\infty} \deg(F^n)^{1/n}$
the \emph{asymptotic degree} of $F$~\cite{RS} and by $\la_2$
the \emph{topological degree} of $F$.  
These degrees are invariant under
conjugacy by polynomial automorphisms
and satisfy $\la_2 \le \la_1^2$. 
In the case $\la_2<\la_1^2$, the Hilbert space
methods of Boucksom and the authors~\cite{deggrowth} 
apply (see also the work of Hubbard-Papadopol~\cite{HP1}, 
Cantat~\cite{cantat} and Manin~\cite{manin}). 
We showed in~\cite{deggrowth} 
that $\deg(F^n) \sim \la_1^n$. Here we use these techniques
to prove that $\nu_*$ attracts all valuations in $\cV_1$ 
(with at most one exception). 

When $\la_2=\la_1^2$, the Hilbert space technique loses some strength.
However, this loss is compensated by the built-in rigidity of these maps.
A more detailed study of the valuative dynamics allows us to show
\begin{thmC}
  Let $F:\C^2\to\C^2$ be a dominant polynomial mapping
  with $\la_2=\la_1^2$.  Then we are in one of the following
  mutually exclusive cases:
  \begin{itemize}
  \item[(1)]
    $\deg(F^n)\sim n\la_1^n$; 
    then $\la_1\in\N$ and in suitable affine coordinates,
    $F$ is a skew product of the form
    $F(x,y)=(P(x),Q(x,y))$, where
    $\deg(P)=\la_1$ and
    $Q(x,y)=A(x)y^{\la_1}+O_x(y^{\la_1-1})$ with 
    $\deg(A) \ge 1$;
  \item[(2)]
    $\deg(F^n)\sim\la_1^n$; 
    then there exists a projective compactification 
    $X\supset\C^2$ with at most quotient
    singularities such that $F$ extends to a holomorphic
    selfmap of $X$.
  \end{itemize}
\end{thmC}
Here, and throughout the paper,
the expression ``in suitable affine coordinates'' 
means that the statement holds after conjugation
by a polynomial automorphism of $\C^2$.

In suitable affine coordinates, $X$ can be chosen as a 
toric surface and we can give normal 
forms for all maps occurring in~(2), 
see Section~\ref{subsecl12l2}. 
When $\la_1$ is an integer, $X$ can be chosen as a 
weighted projective plane. 

Theorem~C extends the 
Friedland-Milnor classification~\cite{FriedlandMilnor}
of polynomial automorphisms with $\la_1=1$. For automorphisms, 
only case (2) appears, and $X$ can be chosen as $\P^2$ or
a Hirzebruch surface.
There is an analogous statement
in the general birational surface case, 
see~\cite{DF,giz}: $\deg(F^n)$ is 
then either bounded or grows linearly or quadratically.

\smallskip
Our last result is another application of the dynamical 
compactifications.
One of the basic problems in the iteration of rational maps is the
construction and study of an ergodic measure of maximal entropy.
When $\la_2>\la_1$, such a measure can be defined as a limit of
preimages of a generic point~\cite{FS2,RS,Sib}, and its basic ergodic
properties are completely understood,
see~\cite{briend-duval,DSallure,DSdecay,guedjannals}.  
In the case of maps with small
topological degree $\la_2<\la_1$, this construction fails. 
A different strategy has been proposed for constructing a dynamically
interesting invariant measure, see~\cite{guedjpoly}.
One first constructs two positive 
closed $(1,1)$ currents, invariant by pull-back 
and by push-forward, respectively.
The measure is then obtained by taking their intersection.
In our setting, the existence of these two currents follows
from Theorem~A, see~\cite{DDG1}. The existence of their
intersection, however, 
is not guaranteed in general except if one has a good control 
of the singularities of their potentials. We prove such a control for
the pull-back invariant current.

Fix affine coordinates $(x,y)$ on $\C^2$ and set 
$\| (x,y) \| = \max \{1, |x|, |y| \}$. 
\begin{thmD}
  Let $F: \C^2 \to \C^2$ be a dominant polynomial mapping with
  $\la_2<\la_1$.  Then the limit
  \begin{equation*}
    G^+(p)=\lim_{n\to\infty}\la_1^{-n}\log^+\|F^np\|
  \end{equation*}
  exists locally uniformly on $\C^2$, and defines a continuous,
  nonzero, non-negative plurisubharmonic function of logarithmic
  growth satisfying $G^+\circ F=\la_1 G^+$.  The support of the
  positive closed current $dd^cG^+$ is equal to $\partial K^+$ where
  $K^+=\{G^+=0\}$.  Further, for each $\e>0$, there exists a constant
  $C>0$, such that
  \begin{equation*}
    \log^+\|F^np\|\le(\la_2+\e)^n(\log^+\|p\|+C)\tag{*}
  \end{equation*}
  for all $n\ge0$, and all $p\in K^+$.
\end{thmD}
Theorem~D generalizes classical properties of the Green function
both of polynomials in one variable and of H{\'e}non maps~\cite{BS1,FS1,HO}. 
On the locus $\{G^+ =0 \}$ the dynamics can
exhibit various speeds of convergence towards infinity, see~\cite{DDS}.
Note that~(*) is in sharp contrast with the phenomenon
described in~\cite{Vigny}. 

\smallskip
The paper is organized as follows. In Sections~1 and~2 we discuss
the relationship between compactifications and valuations, 
and study the induced dynamics on the space of valuations
at infinity. We then turn to the proof
of refined versions of Theorems~A and~B
in the case $\la_2<\la_1^2$: 
in Section~\ref{stab:nondiv} when the eigenvaluation 
is nondivisorial and in Section~\ref{stab:div} 
when it is divisorial. 
Polynomial maps with $\la_2=\la_1^2$ are handled
in Section~\ref{maximum} where we prove Theorem~C.
Theorems~A and~B are proved in Section~\ref{sec:pfAB}, 
Theorem~D in Section~\ref{green}, where we also 
provide a list of examples of maps with $\la_2=\la_1$.
The paper ends with a short appendix outlining an 
adaptation of the necessary material from~\cite{deggrowth} 
to our setting.
\begin{ackn}
  We thank S\'ebastien Boucksom, Serge Cantat 
  and Vincent Guedj for their comments
  on a preliminary version of this paper,
  and the referee for useful suggestions.
\end{ackn}
%
%
%%%%%%%%%%%%%%%%%%%%%%%%%%%%%%%%%%%%%%%%%%%%%%%%%%%%%%%%%%%%%%%%%%%
%
%
\section{Geometry at infinity}\label{geom-infty}
We start by discussing compactifications of $\C^2$ 
together with valuations centered at infinity.
%
%%%%%%%%%%%%%%%%%%%%%%%%%%%%%%%%%%%%%%%%%%%%%%%%%%%%%%%%%%%%%%%%%%%
%
\subsection{Admissible compactifications}
We consider $\C^2$ equipped with a 
fixed embedding into $\P^2$.
\begin{defi}\label{defi:admissible}
  An \emph{admissible compactification} of $\C^2$ is 
  a smooth projective surface $X$ admitting a birational 
  morphism $\pi:X\to\P^2$ that is an isomorphism 
  above $\C^2$.
\end{defi}
It follows that $\pi$ is a composition of point blowups
and that $X\setminus\C^2$ is a connected curve
with simple normal crossings. 
The \emph{primes} of $X$ are the irreducible components
of $X\setminus\C^2$. Every $X$ contains a special prime
$L_\infty$, the strict transform of the line at infinity 
in $\P^2$.
%
%%%%%%%%%%%%%%%%%%%%%%%%%%%%%%%%%%%%%%%%%%%%%%%%%%%%%%%%%%%%%%%%%%%
%
\subsection{Valuations~\cite[Appendix A]{eigenval}} 
Let $R$ be the coordinate ring of $\C^2$.
We define $\hcV_0$ as the set of
valuations $\nu:R\to(-\infty,+\infty]$ \emph{centered at infinity}, 
\ie $\nu(L)<0$ for a generic affine function $L$ on $\C^2$.
If $\nu\in\hcV_0$ and $X$ is an admissible compactification 
of $\C^2$, then the \emph{center} of $\nu$ on $X$ is the
unique scheme-theoretic point on $X$ such that
$\nu$ is strictly positive on the maximal ideal of its
local ring. Thus the center is either a prime of $X$ or
a point on $X\setminus\C^2$. 
We let $\cV_0$ be the subset of $\nu\in\hcV_0$ that are
\emph{normalized} by $\nu(L)=-1$.

There are four kinds of valuations in $\hcV_0$ that we now describe.

First, if $X$ is an admissible compactification of $\C^2$, 
then each prime $E$ of $X$ defines a 
\emph{divisorial} valuation $\ord_E\in\hcV_0$, the order of
vanishing along $E$. In particular, $\ord_{L_\infty}=-\deg$.
Any valuation proportional to some $\ord_E$ will
also be called divisorial.
Thus a valuation is nondivisorial iff its 
center on every admissible compactification is a point.
If we set $b_E=-\ord_E(L)$ for a generic affine function $L$,
then $\nu_E\=b_E^{-1}\ord_E$ is normalized.
We denote by $\hcVdiv$ and $\cVdiv$ 
the set of divisorial valuations in $\hcV_0$ and $\cV_0$,
respectively.

Second, we have \emph{irrational} valuations.
To define them, consider any two primes $E,E'$ in $X$ intersecting 
at a point $p$ and local coordinates $(z,w)$ at $p$ 
such that $E=\{z=0\}$ and $E'=\{w=0\}$. 
To any pair $(s,t)\in\R_+^2$ we attach the valuation $\nu$
defined on the ring $\cO_p$ 
of holomorphic germs at $p$ by 
$\nu(\sum a_{ij}z^iw^j)= \min\{si+tj\ |\ a_{ij}\neq 0\}$;
it does not depend on the choice of coordinates $(z,w)$.
By first extending $\nu$ to 
the common fraction field $\C(X)$ of $\cO_p$ and $R$, 
then restricting to $R$, we obtain a valuation
in $\hcV_0$, called \emph{quasimonomial}.
(It is monomial in the local coordinates $(z,w)$ at $p$.)
The valuation $\nu$ is normalized iff
$sb_E + tb_{E'}=1$.
It is divisorial iff either $t=0$ or the ratio $s/t$ 
is a rational number.
Any divisorial valuation is quasimonomial.
An irrational valuation is by definition a  
nondivisorial quasimonomial valuation.

Third, pick a point $p$ on the line at infinity 
$L_\infty\subset\P^2$,
a formal irreducible curve $C$ at $p$, not contained in 
$L_\infty$, and a constant $\gamma>0$. 
Then $P\mapsto\gamma^{-1}\ord_p(P|_C)$,
for polynomials $P\in R$, defines
a valuation in $\hcV_0$ called a \emph{curve valuation}.
It is normalized iff $\gamma$ is chosen as
the intersection number $(C\cdot L_\infty)$.
Note that a curve valuation can take the value
$+\infty$ when the curve is algebraic.

In suitable affine coordinates $(x,y)$, 
a curve valuation can be computed 
using a Puiseux parameterization $y=h(x^{-1})$ of the curve:
the value on a polynomial $P$ is proportional to
$\ord_{x=\infty}P(x,h(x^{-1}))$.
Now replace the Puiseux series $h$ by a 
formal series of the form
$h(\zeta)=\sum a_k\zeta^{\beta_k}$ with $a_k\in\C^*$, 
and $\beta_k$ an increasing sequence of rational numbers 
with unbounded denominators. 
Then $P\mapsto\ord_{x=\infty}P(x,h(x^{-1}))$ defines
a valuation of the fourth and last kind, 
namely an \emph{infinitely singular valuation}.
%
%%%%%%%%%%%%%%%%%%%%%%%%%%%%%%%%%%%%%%%%%%%%%%%%%%%%%%%%%%%%%%%%%%%
%
\subsection{Tree structure~\cite[Appendix A]{eigenval}} 
The space $\cV_0$ of normalized valuations is equipped with
a partial ordering: $\nu\le\mu$ iff $\nu(P)\le\mu(P)$ 
for all $P\in R$, naturally turning it into
a rooted \emph{tree}. In particular, every two elements 
$\mu,\nu\in\cV_0$ admit a minimum $\mu\wedge\nu\in\cV_0$.
The valuation $-\deg$ is the minimal element of $\cV_0$.
For any admissible compactification $X$, one can map a prime $E$
to the normalized valuation $\nu_E\=b_E^{-1} \ord_E$. 
In this way, we get 
an embedding of the set of primes of $X$ into $\cV_0$. 
The partial ordering coming from $\cV_0$ on the set of primes
coincides with the one coming from the tree structure of the dual 
graph of $X\setminus \C^2$.

The \emph{ends} of $\cV_0$, \ie the maximal elements in the
partial ordering, are the curve and infinitely singular valuations,
\ie the valuations that are not quasimonomial.

We topologize $\cV_0$ by declaring $\nu_n\to\nu$ iff 
$\nu_n(P)\to\nu(P)$ for all $P$. This topology is compact
and admits two important characterizations.
First, it is the weakest topology such that the natural
retraction map $r_I:\cV_0\to I$ is continuous for any
given closed segment $I\subset\cV_0$.
Second, given any admissible compactification $X$
and any point $p\in X\setminus\C^2$, let  
$U(p)\subset\cV_0$ be the set of valuations 
whose center on $X$ is $p$. 
Then the topology on $\cV_0$ is the weakest one in
which $U(p)\subset\cV_0$ is always open.
Each of the four subsets of divisorial, irrational, 
curve and infinitely singular valuations is dense in $\cV_0$. 

There is a unique, decreasing, upper semicontinuous 
\emph{skewness} function $\a:\cV_0\to [-\infty,1]$ satisfying 
$\a(-\deg)=1$ and $|\a(\nu_E)-\a(\nu_{E'})|=(b_Eb_{E'})^{-1}$
whenever $E$ and $E'$ are intersecting primes in some
admissible compactification $X$. See~\cite[{\S}A.1]{eigenval} 
and~\cite[\S6.6.3, \S6.8]{valtree}.
One has $b_E^2\a(\nu_E)\in\Z$, see Lemma~\ref{lem:fundsol}.

Similarly, there is a unique, increasing, lower semicontinuous
\emph{thinness} function $A:\cV_0\to[-2,+\infty]$ such that
$A(\nu_E)=a_E/b_E$, where $a_E=1+\ord_E(dx\wedge dy)$
and $b_E=-\ord_E(L)$ as above, see~\cite[{\S}A.1]{eigenval}.
An irrational valuation has irrational skewness and thinness.
%
%%%%%%%%%%%%%%%%%%%%%%%%%%%%%%%%%%%%%%%%%%%%%%%%%%%%%%%%%%%%%%%%%%%
%
\subsection{The subtree $\cV_1$}\label{sec:V1}
Define $\cV_1$ as the set of valuations $\nu\in\cV_0$ 
with skewness $\a(\nu)\ge0$ and thinness $A(\nu)\le0$. 
Then $\cV_1$ is a subtree of $\cV_0$ of crucial
importance to our study. Its properties are
spelled out in detail in~\cite[Theorem~A.7]{eigenval}.
Here we note the fundamental fact is that any quasimonomial 
valuation in $\cV_1$ is dominated by a \emph{pencil valuation}. 

To define the latter, consider an affine curve 
$C=\{P=0\}\subset\C^2$ with \emph{one place at infinity}, 
that is, the closure of $C$ in $\P^2$ is an irreducible curve 
intersecting the line at infinity in a single point and is 
locally irreducible there. 
Consider the pencil $|C|$ consisting of the affine
curves $C_\la=\{P=\la\}\subset\C^2$ for $\la\in\C$.
It is a theorem by Moh (see~\cite{Moh,CPR}) that
$C_\la$ has one place at infinity for every $\la\in\C$. 
The (normalized) pencil valuation $\nu_{|C|}\in\cV_0$ 
associated to $|C|$ is then defined by 
$\nu_{|C|}(Q)\=(\deg C)^{-1}\ord_\infty(Q|_{C_\la})$, 
for $\la$ generic.
By B{\'e}zout, $-\nu_{|C|}(Q)\deg(C)$ 
equals the number of intersection points in $\C^2$ of 
$C$ with the curve $\{Q=\rho\}$ for generic $\rho\in\C$.
See~\cite[\S A.2]{eigenval} for more details.

This crucial fact that any quasimonomial valuation in $\cV_1$
is dominated by a pencil valuation is the affine
analog of the (easier) 
local result that any quasimonomial valuation
in the valuative tree is dominated by a curve valuation,
see~\cite[Proposition~3.20]{valtree}.

A pencil valuation $\nu_{|C|}$ is divisorial but 
does not itself belong to $\cV_1$ in general. 
Indeed, we have $\a(\nu_{|C|})=0$ and
$A(\nu_{|C|})=(2g-1)/\deg(C)$, where $g$ 
is the geometric genus of the closure of
a generic curve $C_\lambda$ in the pencil.
Thus $\nu_{|C|}$ belongs to $\cV_1$ iff 
$C_\lambda\simeq\C$ for generic $\lambda$. 
We then call $\nu_{|C|}$ \emph{rational pencil valuation}.

It follows that if $\nu$ is a quasimonomial 
valuation in $\cV_1$, then either
$\a(\nu)>0$ or $\nu$ is a rational pencil valuation.
%
%%%%%%%%%%%%%%%%%%%%%%%%%%%%%%%%%%%%%%%%%%%%%%%%%%%%%%%%%%%%%%%%%%%
%
\subsection{Tight compactifications}\label{sec:tight}
Associated to the subtree $\cV_1$ is an important 
class of compactifications of $\C^2$.
\begin{defi}\label{defi:tight}
  An admissible compactification $X$ of $\C^2$ 
  is \emph{tight} if the normalized divisorial 
  valuations associated to the primes of $X$ 
  belong to the subtree $\cV_1$ of $\cV_0$.
\end{defi}
\begin{rmk}
  An admissible compactification $X$ is tight iff 
  $\ord_E(dx \wedge dy)<0$ and $\ord_E(P) \le 0$
  for any prime $E$ of $X$ and 
  any  polynomial $P \in \C[x,y]$. 
  The second condition is 
  equivalent to $Z_{\ord_E}$ being nef, 
  see~\S\ref{sec:classval}; this
  implies that the nef and psef cones of 
  an admissible compactification $X$ of $\C^2$ 
  are simplicial whenever 
  $Z_{\ord_E}$ is nef for every prime $E$ of $X$.
  A compactification of $\C^2$ associated to a curve 
  with one place at infinity as defined in~\cite{CPR} always
  has the latter property and is tight 
  iff the curve is rational.
\end{rmk}
\begin{lem}\label{lem:staytight}
  Let $X$ be a tight compactification 
  of $\C^2$. Pick a point $p\in X\setminus\C^2$
  and let $X'$ be the admissible compactification
  of $\C^2$ obtained by blowing up $p$.
  Then $X'$ is tight iff $p$ does not lie
  on a unique prime of $X$,
  whose associated normalized divisorial valuation has skewness
  zero or thinness zero.
\end{lem}
\begin{proof}
  Let $\nu\in\cV_0$ be the divisorial valuation associated to
  the blowup of $p$. Then $X'$ is tight iff $\nu\in\cV_1$.
  First assume $p$ is the intersection point between
  two primes $E_1$, $E_2$ of $X$
  with associated divisorial valuations $\nu_1,\nu_2\in\cV_0$.
  Then $\nu$ lies in the segment between $\nu_1$ and
  $\nu_2$. Since $\nu_i\in\cV_1$ and $\cV_1$ is a subtree, we get
  $\nu\in\cV_1$, so that $X'$ is tight.
  
  Now assume $p$ lies on a single prime 
  $E$ of $X$, with associated divisorial
  valuation $\nu_E\in\cV_1$. In this case, $\nu>\nu_E$. 
  Let $b_E=-\ord_E(L)$ as above. Then $A(\nu)-A(\nu_E)=1/b_E$ and
  $\a(\nu)-\a(\nu_E)=-1/b_E^2$. On the other hand,
  $b_E^2\a(\nu_E),b_EA(\nu_E)\in\Z$. 
  Hence $\nu\in\cV_1$ iff $A(\nu_E)<0$ and $\a(\nu_E)>0$.
\end{proof}
\begin{cor}\label{cor:tight}
  Let $\nu\in\cV_1$ and
  define a sequence of admissible compactifications 
  $(X_m)_{m\ge0}$ of $\C^2$ as follows: $X_0=\P^2$ and
  $X_{m+1}$ is obtained from $X_m$ by blowing up 
  the center $p_m$ of $\nu$ on $X_m$.
  Then $X_m$ is tight for all $m$.
\end{cor}
\begin{proof}
  In view of Lemma~\ref{lem:staytight} we only have to show 
  that $p_m$ never lies on a unique prime $E$ of $X_m$ 
  whose associated normalized
  valuation $\nu_E\in\cV_0$ has skewness zero or thinness zero. 
  But if it did, the valuation $\nu_k\in\cV_0$ associated to
  $p_k$ would satisfy $\nu_k>\nu_E$ for all $k>m$.
  This would imply $\nu>\nu_E$ so that $\a(\nu)<0$ or $A(\nu)>0$, 
  contradicting $\nu\in\cV_1$.
\end{proof}
%
% 
%%%%%%%%%%%%%%%%%%%%%%%%%%%%%%%%%%%%%%%%%%%%%%%%%%%%%%%%%%%%%%%%%%%
%
%
\section{Valuative dynamics}\label{val-dyn}
Consider a dominant polynomial mapping $F:\C^2\to\C^2$. 
%
%%%%%%%%%%%%%%%%%%%%%%%%%%%%%%%%%%%%%%%%%%%%%%%%%%%%%%%%%%%%%%%%%%%
%
\subsection{Induced map on valuations}\label{sec:indmap}
Let $\nu\in\hcV_0$ be a valuation centered at infinity 
and set $d(F,\nu)\=-\nu(F^*L)\ge0$ for a generic 
affine function $L$ on $\C^2$. 
Define a valuation $F_*\nu$ by
$F_*\nu=0$ if $d(F,\nu)=0$ and 
$F_*\nu(P)=\nu(F^*P)$ if $d(F,\nu)>0$.
In the former case, note that $\nu(F^*P)=0$ for a generic polynomial.
In the latter case, $F_\bullet\nu\=F_*\nu/d(F,\nu)$ is a well-defined
normalized valuation in $\cV_0$.

We have the following valuative criterion for properness of maps.
It is a consequence of~\cite[Proposition~7.2]{eigenval}.
\begin{prop}\label{prop:proper-crit}
  When $F$ is not proper, one can find a divisorial valuation $\nu$
  such that $d(F,\nu) =0$. 
  When $F$ is proper, there exists a constant $c>0$ 
  such that $d(F,\nu) \ge c$ for all $\nu \in \cV_0$.
\end{prop}
When $\nu\in\hcVdiv$ is divisorial and $d(F,\nu)>0$, the
valuation $\nu'\=F_*\nu\in\hcVdiv$ is also divisorial. 
More precisely, given any two admissible compactifications 
$X$, $X'$ of $\C^2$ such that the centers of $\nu$ and $\nu'$
are primes $E$ and $E'$ of $X$ and $X'$, respectively
we can write $\nu=t\ord_E$ and $\nu'=t'\ord_{E'}$.
Then $t'/t$ is the coefficient of $E$ in $\tF^*E'$,
where $\tF:X\dashrightarrow X'$ is the lift of $F$.

The following result, proved in~\cite[\S 7.2]{eigenval}
allows us to do dynamics on the subtree $\cV_1$ of $\cV_0$,
even when $F$ is not proper.
\begin{prop}\label{prop:V1invariant}
  We have $d(F,\cdot)>0$ on $\cV_1$ and 
  $F_\bullet$ leaves $\cV_1$ invariant.
\end{prop}
The map $F_\bullet$ preserves the tree structure on 
$\cV_1$ in the sense that small segments are mapped
homeomorphically onto small segments. 
See~\cite[Theorem~7.4]{eigenval} for a precise statement.
%
%%%%%%%%%%%%%%%%%%%%%%%%%%%%%%%%%%%%%%%%%%%%%%%%%%%%%%%%%%%%%%%%%%%
%
\subsection{Eigenvaluations}\label{sec:eigenval}
Define the \emph{asymptotic degree} of $F$ by
$\la_1\=\lim_{n\to\infty}(\deg F^n)^{1/n}$.
The following result was proved in~\cite[\S 7.3]{eigenval}.
\begin{prop}\label{prop:eigenval}
  There exists a valuation $\nu_*$ belonging to the
  subtree $\cV_1$ of $\cV_0$ such that $F_*\nu_*=\la_1\nu_*$. 
  In particular, $\a(\nu_*)\ge0$ and $d(F,\nu_*)=\la_1$.
\end{prop}
Such a valuation is called an \emph{eigenvaluation}.
The proof is based on the fact that $F_\bullet$
preserves the tree structure on $\cV_1$.
We also proved that the eigenvaluation 
admits a small basin of attraction. 
Using the techniques of~\cite{deggrowth} as described in
the appendix, we can strengthen these conclusions considerably
under the assumption $\la_2<\la_1^2$, where 
$\la_2$ is the topological degree of $F$.
\begin{thm}\label{thm:basin}
  Assume $\la_2<\la_1^2$.
  \begin{itemize}
  \item[(a)]
    the valuation $\nu_*$ in 
    Proposition~\ref{prop:eigenval}
    is the unique valuation $\nu\in\cV_0$ with
    $\a(\nu)\ge0$ and $F_*\nu=\la_1\nu$;
  \item[(b)]
    if $\nu\in\cV_0$ and $\a(\nu)>0$, then
    $F^n_\bullet\nu\to\nu_*$ in $\cV_0$ as $n\to\infty$;
  \item[(c)]
    there exists at most one $\nu\in\cV_0$ with $\a(\nu)=0$
    such that $F^n_\bullet\nu\not\to\nu_*$ as $n\to\infty$;
    this $\nu$ must satisfy 
    $F_\bullet\nu=\nu$.
  \end{itemize}
\end{thm}
Hence it makes sense to refer to $\nu_*$
as \emph{the} eigenvaluation when $\la_2<\la_1^2$.
\begin{proof}
  The proof invokes the spectral properties
  of the operators $F^*$ and $F_*$ on the space 
  $\Ltwo(\fX)$ as discussed in Appendix~A.

  For $\nu\in\hcV_0$ let $Z_\nu\in W(\fX)$ be the associated
  Weil class, see Section~\ref{sec:classval}. 
  Then $F_*Z_\nu=Z_{F_*\nu}$. 
  When $\nu\in\cV_0$ and $\a(\nu)\ge0$,
  Lemma~\ref{lem:nefcrit} shows 
  $Z_\nu$ is nef, hence in $\Ltwo(\fX)$,
  Theorem~\ref{thm:spectral} therefore implies that 
  $F_*\nu=\la_1\nu$ iff $Z_\nu=c\,\theta_*$ for some $c>0$. 
  In fact, $c=1$ since $(\theta_*\cdot\cL)=(Z_\nu\cdot\cL)=1$,
  and so $\nu=\nu_*$. This proves~(a), and that 
  $\theta_*=Z_{\nu_*}$.

  To prove~(b) and~(c), pick any $\nu\in\cV_0$ with 
  $\a(\nu)\ge0$. It follows from 
  Theorem~\ref{thm:spectral} that
  $\la_1^{-n}F^n_*Z_\nu\to(Z_\nu\cdot\theta^*)\theta_*
  =(Z_\nu\cdot\theta^*)Z_{\nu_*}$
  as $n\to\infty$. 
  Thus $F^n_\bullet\nu\to\nu_*$ as long as 
  $(Z_\nu\cdot\theta^*)>0$. 
  Since $Z_\nu$ is nef and 
  $(\theta^*\cdot\theta^*)=0$,
  the Hodge inequality implies 
  $(Z_\nu\cdot\theta^*)\ge0$ with equality iff
  $Z_\nu=c\theta^*$ for some $c>0$.
  In the latter case, $\nu$ is uniquely determined,
  $\a(\nu)=(Z_\nu\cdot Z_\nu)=0$ and
  $F_\bullet\nu=\nu$ as $F_*\theta^*=(\la_2/\la_1)\theta^*$.
\end{proof}
\begin{prop}\label{prop:quadint}
  The asymptotic degree $\la_1$ is a quadratic integer.
  Moreover, when $\la_2<\la_1^2$, we are in
  one of the following two cases:
  \begin{itemize}
  \item[(i)]
    $\la_1\in\N$ and $\nu_*$ is divisorial or infinitely singular;
  \item[(ii)]
    $\la_1\not\in\Q$ and $\nu_*$ is irrational.
  \end{itemize}
\end{prop}
\begin{proof}
  The fact that $\la_1$ is a quadratic integer 
  is contained in~\cite[Theorem~A']{eigenval}.
  An outline of the proof goes as follows. 

  When $\nu_*$ is divisorial, it follows from
  the discussion above Proposition~\ref{prop:proper-crit} 
  that $\la_1=d(F,\nu_*)\in\N$.

  When $\nu_*$ is not divisorial, 
  Theorem~7.7 in~\cite{eigenval} provides local 
  normal forms of $F$ at some point at infinity 
  of some suitable admissible compactification
  of $\C^2$. In the infinitely 
  singular case, one sees 
  by inspection that $\la_1$ is an integer. 

  If instead $\nu_*$ is irrational, then 
  the local normal form at $p$ is monomial, and 
  $\la_1$ is the spectral radius of a $2\times2$ 
  matrix $M$ having nonnegative integer coefficients.
  Suppose $\la_1\in\Q$.
  Then the other eigenvalue $\la'_1$
  of $M$ is also rational. 
  Now $\nu_*$ being irrational means $M$ has
  an eigenvector $(u,v)\in\R_+^2$ 
  with $u/v$ irrational.
  This is only possible
  if $\la'_1=\la_1$. But then the local
  topological degree of $F$ at $p$ equals 
  $\det M=\la_1^2$, contradicting $\la_2<\la_1^2$.
\end{proof}
\begin{prop}\label{prop:nondiv}
  When $\la_2<\la_1$, the eigenvaluation $\nu_*$ cannot
  be divisorial.
\end{prop}
\begin{proof}
  Assume $\nu_*$ is divisorial. 
  Pick an admissible compactification $X$ of $\C^2$ 
  such that $\nu_*$ is proportional to $\ord_E$ for
  some prime $E$ of $X$.
  Then the rational lift $\tF:X\dashrightarrow X$
  maps $E$ onto itself, and the eigenvalue $\la_1$
  is the coefficient of $E$ in $\tF^*E$.
  This coefficient is dominated by the topological 
  degree of $F$ in a neighborhood of $E$.
  Hence $\la_1\le\la_2$.
\end{proof}
In the proof of Theorem~A, the case when $\nu_*$ is
divisorial and also an end in $\cV_1$ needs special treatment.
It occurs exactly when $\a(\nu_*)=0>A(\nu_*)$, that is,
$\nu_*$ is a rational pencil valuation; or when
$\a(\nu_*)>0=A(\nu_*)$. 
As the next results show, the dynamics 
is then quite particular.
\begin{prop}\label{prop:skewproduct}
  Assume that $\nu\in\cV_1$ is a rational pencil 
  valuation such that $F_*\nu=\la\nu$ for some
  $\la>0$. Then $\la$ is an integer and
  $F$ is conjugate by a polynomial automorphism 
  of $\C^2$ to a skew product of the form
  $F(x,y)=(P(x),Q(x,y))$ with $\deg_yQ=\la$.
\end{prop}
\begin{proof}
  As in~\cite[\S 7.4]{eigenval} this follows 
  from the Line Embedding Theorem.
\end{proof}
\begin{prop}\label{prop:jacobian}
  Assume $\la_1>1$. If there exists a valuation 
  $\nu_*\in\cV_0$ 
  (not necessarily divisorial) with
  $F_*\nu_*=\la_1\nu_*$ and
  $\a(\nu_*)>0=A(\nu_*)$, then 
  $F$ is a counterexample to the Jacobian conjecture.
\end{prop}
\begin{proof}
  The change of variables formula implies 
  $A(F_*\nu)=A(\nu)+\nu(JF)$ for all 
  $\nu\in\cV_1$, where $JF$ is the Jacobian determinant
  of $F$, see~\cite[Lemma~7.6]{eigenval}.
  Applying this to $\nu=\nu_*$ yields $\nu_*(JF)=0$. 
  As $\a(\nu_*)>0$, this is only possible if 
  $JF$ is constant. 
  But if $F$ were a polynomial automorphism, we would have
  $\a(\nu_*)=(\theta_*(F)\cdot\theta_*(F))
  =(\theta^*(F^{-1})\cdot\theta^*(F^{-1}))=0$.
\end{proof}
%
%%%%%%%%%%%%%%%%%%%%%%%%%%%%%%%%%%%%%%%%%%%%%%%%%%%%%%%%%%%%%%%%%%%
%
\subsection{Examples of eigenvaluations}\label{sec:exameig}
Curve valuations do not belong to $\cV_1$, hence can never
be eigenvaluations. As the following examples show,
any other type of valuation can occur.
\begin{exam}
  Assume that the extension to $\P^2$ of
  $F:\C^2\to\C^2$ does not contract the line at infinity. 
  Then the divisorial valuation $-\deg$ is an eigenvaluation.
\end{exam}
\begin{exam}
  Any rational pencil valuation appears as an eigenvaluation.
  Indeed, in suitable affine coordinates, the valuation is
  associated to the pencil $x=\mathrm{const}$ and is the 
  eigenvaluation of a suitable skew product
  $F(x,y)=(P(x),Q(x,y))$.
\end{exam}
\begin{exam}
  Pick positive integers $a,b,c,d$ such that 
  $\Delta\=(a+d)^2-4(ad-bc)>0$. 
  Then the $2\times2$ matrix $M$ with entries $a,b,c,d$ 
  has two real eigenvalues $t>1$ and $t'<t$. 
  The eigenvalue $t$ admits an eigenvector
  $(u,v)\in\R^2_-$ that we can normalize 
  by the condition $\min\{u,v\}=-1$.
  The map $F(x,y) \= (x^ay^b,x^cy^d)$ has 
  topological degree $\la_2=|ad-bc|$ and
  asymptotic degree $\la_1=t$.
  It admits a unique eigenvaluation which is 
  the monomial valuation with weights
  $u$ on $x$ and $v$ on $y$. When $\Delta$ is not a square, 
  the eigenvaluation is irrational.
  Otherwise it is divisorial.

  One may also perturb $F$ by adding terms of 
  sufficiently low order. After doing so, 
  $\la_1,\la_2$ and the eigenvaluation remain unchanged.
\end{exam}
\begin{exam}
  The eigenvaluation $\nu_*$ of 
  a polynomial automorphism $F$
  with $\la_1>1$ is always 
  infinitely singular. Indeed, 
  the eigenclass $\theta_*\in\Ltwo(\fX)$
  can be written $\theta_*(F)=\theta^*(F^{-1})$,
  hence 
  $\a(\nu_*)=(\theta_*(F)\cdot\theta_*(F))=0$.
  Thus $\nu_*$ cannot be irrational,
  as that would imply $\a(\nu_*)\not\in\Q$.
  It cannot be divisorial by Proposition~\ref{prop:nondiv}.
  Hence it is infinitely singular. 
  Proposition~\ref{prop:quadint} now implies
  that the asymptotic degree $\la_1$ is 
  an integer, a fact which also follows from the 
  Friedland-Milnor classification~\cite{FriedlandMilnor}.

  It would be interesting to investigate 
  the relation between the eigenvaluation $\nu_*$,
  the solenoids constructed by 
  Hubbard et al.~\cite{HP1,HPV}, and
  the singularity of the Green function of $F$
  at infinity; see also~\cite[Proposition~6.9]{pshsing}.
\end{exam}
\begin{exam}
  The argument above shows more generally that 
  the eigenvaluation of 
  any polynomial mapping $F:\C^2\to\C^2$ 
  with $\la_2<\la_1\in\N$ must be infinitely singular.
  The condition of $\la_1$ being an integer is
  satisfied, for instance, if $F$ defines an
  algebraically stable map on $\P^2$, in
  which case $\la_1=\deg F$~\cite{FS2}.
\end{exam}
%
%
%%%%%%%%%%%%%%%%%%%%%%%%%%%%%%%%%%%%%%%%%%%%%%%%%%%%%%%%%%%%%%%%%%%
%
%
\section{Stability when $\la_2<\la_1^2$: the nondivisorial case}\label{stab:nondiv}
Our aim is now to prove precise versions 
of Theorem~A and~B in the case when the eigenvaluation is 
nondivisorial. Recall that we are always in this situation
when $\la_2<\la_1$, see Proposition~\ref{prop:nondiv}.
\begin{thm}\label{thm:stability1}
  Let $F:\C^2\to\C^2$ be a dominant polynomial mapping 
  with $\la_2<\la_1^2$ such that the eigenvaluation 
  $\nu_*$ is nondivisorial.
  Then there exists a tight compactification 
  $X$ of $\C^2$, a point $p\in X\setminus\C^2$
  and local coordinates $(z,w)$ at $p$ such that 
  the lift $\tF:X\dashrightarrow X$ 
  defines a superattracting fixed point 
  germ at $p$ taking one of the following forms:
  \begin{itemize}
  \item[(a)]
    $\tF(z,w)=(z^aw^b,z^cw^d)$, where $a,b,c,d\in\N$
    and the $2\times2$ matrix $M$ with entries $a,b,c,d$
    has spectral radius $\la_1$;
    locally at $p$ we have $X\setminus\C^2=\{zw=0\}$;
  \item[(b)]
    $\tF(z,w)=(z^{\la_1},\mu z^cw+P(z))$,
    where $c\ge1$, $\mu\in\C^*$ and  
    $P\not\equiv P(0)=0$ is a polynomial;
    locally at $p$ we have $X\setminus\C^2=\{z=0\}$.
  \end{itemize}
  Moreover, if $F$ is not conjugate to a skew product 
  by a polynomial automorphism of $\C^2$, 
  there exists $n\ge 1$ such that each prime of $X$ 
  is contracted to $p$ by $\tF^n$.
  When $F$ is conjugate to a skew product, the same conclusion
  holds for all primes of $X$ with the exception of a
  single prime invariant by $\tF$.
\end{thm}
\begin{rmk}
  The skew product case in Theorem~\ref{thm:stability1}
  does not occur when $\la_2<\la_1$. 
  Indeed, if $F(x,y)=(P(x),Q(x,y))$, then 
  $\la_2=\deg P\cdot\deg_yQ\ge\max\{\deg P,\deg_yQ\}=\la_1$.
\end{rmk}
\begin{cor}\label{cor:nondivrec}
  Under the assumptions of Theorem~\ref{thm:stability1},
  for any $\nu\in\cV_1$, there exists $n=n(\nu)$
  such that the sequence $(d(F^j,\nu))_{j\ge n}$ satisfies an
  integral linear recursion formula of order 1 or 2.
  In particular, $(\deg F^j)_{j\ge n}$ satisfies
  such a recursion formula for $n$ large enough.
\end{cor}
\begin{proof}[Proof of Theorem~\ref{thm:stability1}]
  Since $\nu_*$ is nondivisorial, its center on 
  any admissible compactification of $\C^2$ is a point. 
  We may therefore define an infinite sequence 
  $(X_m)_{m\ge0}$ of admissible compactifications 
  of $\C^2$ by setting $X_0=\P^2$ and letting
  $X_{m+1}$ be the blowup of $X_m$ at the center
  $p_m\in X_m$ of $\nu_*$. 
  Let $\nu_0,\dots,\nu_m$ be the normalized divisorial
  valuations associated to the primes of $X_m$.  
  As the eigenvaluation $\nu_*$ lies in $\cV_1$,
  the compactification $X_m$ is tight by 
  Corollary~\ref{cor:tight}.
  Hence, for any $j$, either $\a(\nu_j)>0$ or
  $\nu_j$ is a rational pencil valuation.
  
  We claim that there exists $m\ge0$ such that 
  the lift $\tF_m:X_m\dashrightarrow X_m$ defines
  a superattracting fixed point germ at $p_m$,
  taking one of the forms~(a) or~(b) above.
  This follows from the proof of 
  Theorem~7.7 in~\cite{eigenval}; let us briefly 
  indicate how to proceed.
  
  When $\nu_*$ is infinitely singular, 
  there exists an infinite subsequence $(m_j)_{j\ge1}$ 
  such that $X_{m_j}\setminus\C^2$ is locally irreducible
  at $p_{m_j}$. The corresponding
  valuations $\nu_{m_j+1}$ increase to $\nu_*$ as
  $j\to\infty$. 
  Consider the segment $I_j\=\,]\nu_{m_j+1},\nu_*]$ in $\cV_1$
  and the corresponding open neighborhood 
  $U_j$ of $\nu_*$ in $\cV_0$ 
  consisting of valuations whose tree retraction to
  the closed segment $\overline{I_j}$ belongs to $I_j$.
  The attracting properties of $\nu_*$
  imply $F_\bullet I_j\subset I_j$ and
  $F_\bullet U_j\subset\subset U_j$ for large $j$.
  Now $U_j$ is the set of valuations whose center on 
  $X_{m_j}$ equals $p_{m_j}$, so this means that
  $\tF_{m_j}$ defines a holomorphic fixed point germ at $p_{m_j}$
  which in fact is rigid in the sense of~\cite{F-rigid}.
  The normal form in~(b) follows from
  the classification in~\cite{F-rigid}. 
  A direct inspection shows that $\tF_{m_j}$ is superattracting 
  at $p_{m_j}$.
  
  If instead $\nu_*$ is irrational, then for all large $m$, 
  $p_m$ is the intersection of two primes of $X_m$.
  Now $F_\bullet$ preserves the tree structure on $\cV_0$:
  if $I_0$ is any small open segment in $\cV_0$ containing $\nu_*$ 
  and $I$ a small subsegment, then
  $F_\bullet$ is a homeomorphism of $I$ onto its image
  $F_\bullet I\subset I_0$.
  Moreover, $F|_{I_0}$ is contracting at $\nu_*$ 
  in the skewness metric. This means that 
  if $I$ is sufficiently symmetric around $\nu_*$,
  then $F_\bullet I\subset I$.
  (The symmetry condition is only necessary when $F_\bullet$
  is order-reversing on $I_0$.)
  Let $\mu_m,\mu'_m\in\cV_0$ be the normalized divisorial valuations
  associated to the primes of $X_m$ containing $p_m$ and
  set $I_m=\,]\mu_m,\mu'_m[$. 
  By repeating the arithmetic argument as
  in~\cite[Lemma~5.6]{eigenval}, one shows that there exists
  an infinite subsequence $(m_j)_{j\ge1}$ such that $I_{m_j}$
  is sufficiently symmetric so that 
  $F_\bullet I_{m_j}\subset I_{m_j}$. 
  For $j$ large enough we will also have 
  $F_\bullet U_{m_j}\subset U_{m_j}$,
  where $U_m$ is the set of valuations whose center on $X_m$
  equals $p_m$. Thus $\tF_{m_j}$ defines
  a holomorphic fixed point germ at $p_{m_j}$, which by
  invoking~\cite{F-rigid} can be put in the monomial 
  form~(a) above.
  Finally $\tF_{m_j}$ must be superattracting at 
  $p_{m_j}$ or else one of the primes of $X_{m_j}$ 
  containing $p_{m_j}$ would be
  an eigenvaluation for $F^2$ with eigenvalue $\la_1^2$, 
  contradicting Theorem~\ref{thm:basin}~(a).
  
  In both the irrational and infinitely singular case we
  have found a tight compactification $X$ of $\C^2$
  and a point $p\in X\setminus\C^2$ such that the lift 
  $\tF:X\dashrightarrow X$ defines a superattracting fixed
  point germ at $p$. 
  The point $p$ defines an open subset $U(p)$ of $\cV_0$:
  a valuation $\nu\in\cV_0$ belongs to $U(p)$ iff the 
  center of $\nu$ on $X$ equals $p$. 
  
  Now pick any prime $E$ of $X$
  and let $\nu_E\in\cV_0$ be the associated divisorial valuation.
  We have $\a(\nu_E)\ge0$ since $X$ is tight.
  If $\a(\nu_E)>0$ then Theorem~\ref{thm:basin}~(b) shows
  that $F^n_\bullet\nu_E\in U(p)$ for $n\gg1$. This means that
  $\tF^n$ contracts $E$ to $p$.
  If instead $\a(\nu_E)=0$, then $\nu_E$ 
  must be a rational pencil valuation.
  Theorem~\ref{thm:basin}~(c) shows that $\tF^n$ still
  contracts $E$ to $p$ for $n\gg1$ unless $F_\bullet\nu_E=\nu_E$.
  In the latter case, $F$ is conjugate to a skew product
  by Proposition~\ref{prop:skewproduct}.
\end{proof}
\begin{proof}[Proof of Corollary~\ref{cor:nondivrec}]
  Write $\nu_j=F^j_\bullet(\nu)$ for $j\ge0$. 
  Then $d(F^j,\nu)=\prod_{i=0}^{j-1}d(F,\nu_i)$.
  We may assume $F^n_\bullet\nu\to\nu_*$ as $n\to\infty$,
  since otherwise $F_\bullet\nu=\nu$ and 
  $d(F^j,\nu)=d(F,\nu)^j$.  
  
  If $\nu_*$ is infinitely singular, then 
  $d(F,\cdot)\equiv\la_1$ in a neighborhood of 
  $\nu_*$. Since $\nu_j\to\nu_*$ as $j\to\infty$,
  we get $d(F^{j+1},\nu)=\la_1d(F^j.\nu)$ for $j\gg0$.
 
  When instead $\nu_*$ is irrational, we use the local
  monomial form in Theorem~\ref{thm:stability1}~(b).
  Let $E_z=\{z=0\}$ and $E_w=\{w=0\}$ be the primes
  of $X$ containing $p$ and write 
  $b_z=-\ord_{E_z}(L)$, $b_w=-\ord_{E_w}(L)$
  for a generic affine function $L$ on $\C^2$.
  For $j\ge n$, set
  $s_j\=F^j_*\nu(z)>0$, 
  $t_j\=F^j_*\nu(w)>0$.
  Then $(s_{j+1},t_{j+1})=M(s_j,t_j)$.
  Now $d(F^j,\nu)=b_zs_j+b_wt_j$, so this
  easily implies that 
  $(d(F^j,\nu))_{j\ge n}$
  satisfies an integral linear recursion formula
  of order at most two.
\end{proof}
%
%
%%%%%%%%%%%%%%%%%%%%%%%%%%%%%%%%%%%%%%%%%%%%%%%%%%%%%%%%%%%%%%%%%%%
%
%
\section{Stability when $\la_2<\la_1^2 $: 
the divisorial case}\label{stab:div}
Next we prove Theorem~A and~B in the case when the 
eigenvaluation $\nu_*$ is divisorial. Recall that this
implies $\la_2\ge\la_1$. 
We distinguish between two subcases: 
$\nu_*$ may or may not be an end in the tree $\cV_1$.

When $\nu_*$ is an end, either
$F$ is conjugate to a skew product; or $F$ is 
a counterexample to the Jacobian conjecture,
see Propositions~\ref{prop:skewproduct}
and~\ref{prop:jacobian}.
\begin{thm}\label{thm:stability2}
  Let $F:\C^2\to\C^2$ be a dominant polynomial mapping 
  with $\la_2<\la_1^2$. Assume that the eigenvaluation 
  $\nu_*$ is divisorial and an end in $\cV_1$.
  Then there exists a tight compactification 
  $X$ of $\C^2$, a prime $E_*$ of $X$, a point $p$
  on $E_*$, and an integer $n\ge1$ 
  such that the lift 
  $\tF:X\dashrightarrow X$ maps $E_*$ onto $E_*$
  and defines a holomorphic fixed point germ at $p$;
  and we are in one of the following situations:
  \begin{itemize}
  \item[(i)]
    for each prime $E\ne E_*$ 
    of $X$, either $\tF^n(E)=E_*$, 
    or $\tF^n$ contracts $E$ to $p$;
  \item[(ii)]
    $F$ is conjugate to a skew product 
    by a polynomial automorphism of $\C^2$
    and the properties in~(i) hold for all
    primes $E$ of $X$ with the exception of a
    single prime invariant by $\tF$.
  \end{itemize}
\end{thm}
\begin{cor}\label{cor:divendrec}
  Under the assumptions of Theorem~\ref{thm:stability2},
  for any $\nu\in\cV_1$, there exists $n=n(\nu)$
  such that the sequence $(d(F^j,\nu))_{j\ge n}$ satisfies an
  integral linear recursion formula of order 1 or 2.
  In particular, $(\deg F^j)_{j\ge n}$ satisfies
  such a recursion formula for $n$ large enough.
\end{cor}
When $\nu_*$ is not an end, so that $\a(\nu_*)>0>A(\nu_*)$,
the result is slightly less precise and the proof
more subtle.
\begin{thm}\label{thm:stability3}
  Let $F:\C^2\to\C^2$ be a dominant polynomial mapping 
  with $\la_2<\la_1^2$. Assume that the eigenvaluation 
  $\nu_*\in\cV_1$ is divisorial and not an end in $\cV_1$.
  Then there exists a tight compactification $X$ of
  $\C^2$, a prime $E_*$ of $X$ and an integer $n\ge1$ 
  such that the lift 
  $\tF:X\dashrightarrow X$ maps $E_*$ onto $E_*$,
  and we are in one of the following situations:
  \begin{itemize}
  \item[(i)]    
    for each prime $E\ne E_*$ of $X$,
    either $\tF^n(E)=E_*$, 
    or $\tF^n$ contracts $E$ to a point
    on $E_*$ at which all iterates of $\tF$ are holomorphic;
  \item[(ii)]    
    $F$ is conjugate to a skew product 
    by a polynomial automorphism of $\C^2$
    and the properties in~(i) hold for all
    primes $E$ of $X$ with the exception of a
    single prime invariant by $\tF$.
  \end{itemize}
\end{thm}
\begin{cor}\label{cor:divnonendrec}
  Under the assumptions of Theorem~\ref{thm:stability3},
  there exist $l\ge1$ and, for any $\nu\in\cV_1$, 
  $n=n(\nu)$ such that 
  the sequence $(d(F^{lj+k},\nu)_{j\ge 0}$
  satisfies an integral linear recursion formula of 
  order 1 or 2 for any $k\ge n$.
  In particular, $(\deg F^{lj+k})_{j\ge 0}$ satisfies
  such a recursion formula for any sufficiently large $k$.
\end{cor}
\begin{exam}\label{ex:recursion}
  Define $F(x,y)=(x(x-y^2),x+y)$ and set $a_j=\deg(F^j)$.
  Then $a_0=1$, $a_1=3$, $a_2=6$ and 
  $a_j=a_{j-1}+a_{j-2}+2a_{j-3}$ for $j\ge3$.
  One checks that $\la_1=2$, $\la_2=3$, and that
  $(a_j)_0^\infty$ does not satisfy any integral 
  recursion formula of order smaller than three.
\end{exam}
\begin{proof}[Proof of Theorem~\ref{thm:stability2}]
  The proof is similar to the infinitely
  singular case of Theorem~\ref{thm:stability1}.
  Let $X_0$ be the smallest admissible compactification
  such that the center of $\nu_*$ on $X_0$ is a prime 
  $E_*$ of $X_0$: it is obtained from $\P^2$ by
  successively blowing up the center of $\nu_*$, so
  $X_0$ is tight by Corollary~\ref{cor:tight}.
  Inductively define a sequence $(X_m)_{m\ge0}$ 
  of tight compactifications by letting 
  $X_{m+1}$ be the blowup of $X_m$ at the unique 
  intersection point of (the strict transform of) $E_*$
  and the unique prime $E_m\ne E_*$ of $X_m$
  intersecting $E_*$.

  Let $\nu_m\in\cV_0$ be the divisorial valuation associated 
  to $E_m$. Then $(\nu_m)_{m\ge0}$ form a sequence of
  valuations increasing to $\nu_*$. 
  Set $I_m=\,]\nu_m,\nu_*[$. 
  As in~\cite[\S 7.3]{eigenval} we see that $F_\bullet$
  maps $I_m$ into itself for large $m$.
  Moreover, we have $F_\bullet U_m\subset U_m$, where
  $U_m\subset\cV_0$ is the set of valuations 
  whose center on $X_m$ equals $p_m$. 
  Set $X=X_m$ and $p=p_m$. 
  Then the lift $\tF:X\dashrightarrow X$
  maps $E_*$ onto itself and defines a holomorphic 
  fixed point germ at~$p$.

  Let $U$ be the open neighborhood of $\nu_*$ in 
  $\cV$ consisting of valuations whose center on $X$ is 
  contained in $E_*$. 
  Pick a prime $E\neq E_*$ of $X$. If $F_\bullet\nu_E=\nu_E$, then
  $E$ is unique with this property, $\nu_E$ is a 
  rational pencil valuation and $F$ is conjugate to 
  a skew product. For all other primes $E$ we have 
  $F^n_\bullet\nu_E\in U$ for $n\gg1$. 
  If $F^n_\bullet\nu_E=\nu_*$, then $\tF^n(E)=E_*$.
  Otherwise $E$ is contracted by $\tF^n$ to $p\in E_*$.
  This completes the proof.
\end{proof}
 \begin{proof}[Proof of Corollary~\ref{cor:divendrec}]
   We can use the same proof as of
   Corollary~\ref{cor:nondivrec} in the irrational case.
   Indeed, let $E$ and $E_*$ 
   be the primes of $X$ containing $p$
   and write $E=\{z=0\}$, $E_*=\{w=0\}$ for
   coordinates $(z,w)$ at $p$.
   We may perhaps not arrange that $F$ is 
   monomial in $(z,w)$, but if we set 
   $s_j\=F^j_*\nu(z)>0$, 
   $t_j\=F^j_*\nu(w)>0$,
   we will nevertheless have
   $(s_{j+1},t_{j+1})=M(s_j,t_j)$ 
   for some $2\times2$ matrix $M$ with nonnegative 
   integer entries, see~\cite[Theorem~7.4]{eigenval} 
 \end{proof}
\begin{proof}[Proof of Theorem~\ref{thm:stability3}]
  Since $\nu_*$ is not an end in $\cV_1$, we have
  $\a(\nu_*)>0>A(\nu_*)$.
  Let $X_0$ be the smallest admissible compactification
  such that the center of $\nu_*$ on $X_0$ is a prime 
  $E_*$ of $X_0$: it is obtained from $\P^2$ by
  successively blowing up the center of $\nu_*$, so
  $X_0$ is tight by Corollary~\ref{cor:tight}.
  \begin{lem}\label{lem:quotient}
    There exists a tight compactification $X$ of $\C^2$
    dominating $X_0$ 
    such that the lift $\tF:X\dashrightarrow X$ 
    is holomorphic at all periodic points of 
    $\tF|_{E_*}$, where $E_*$ denotes the center of
    $\nu_*$ on $X$.
  \end{lem}
  Using this lemma we now conclude the proof.
  Let $U$ be the open neighborhood of $\nu_*$ in 
  $\cV$ consisting of valuations whose center on $X$ is 
  contained in $E_*$. 
  Pick a prime $E\neq E_*$ of $X$. If $F_\bullet\nu_E=\nu_E$, then
  $E$ is unique with this property, $\nu_E$ is a 
  rational pencil valuation and $F$ is conjugate to 
  a skew product. For all other primes $E$ we have 
  $F^n_\bullet\nu_E\in U$ for $n\gg1$. 
  If $F^n_\bullet\nu_E=\nu_*$, then $\tF^n(E)=E_*$.
  Otherwise $E$ is contracted by $\tF^n$ to a point $p\in E_*$.
  By increasing $n$ we can assume that the orbit of $p$
  under $\tF|_{E_*}$ is either periodic or infinite. 
  In the first case, $\tF$ is holomorphic at $p$ by
  Lemma~\ref{lem:quotient}. In the second case, we can
  by increasing $n$ assume that the orbit does not intersect
  the indeterminacy set of $\tF$.
  This completes the proof.
\end{proof}
\begin{proof}[Proof of Lemma~\ref{lem:quotient}]
  Write $\tF_0$ for the lift of $F_0$ to $X_0$.
  Let $Z\subset E_*$ be the (finite) set of periodic 
  points of $\tF_0|_{E_*}$ whose orbit contains an indeterminacy 
  point of $\tF_0$. 
  First assume for simplicity that $Z$
  consists of a single periodic orbit 
  $p_0$, $p_1$,\dots, $p_l=p_0$ 
  of length $l\ge1$ for $\tF_0|_{E_*}$.

  Let $\mu_k\in\cV_0$ be the divisorial valuation 
  associated to the blowup of $X_0$ at $p_k$, $0\le k\le l$.
  The segment $J_k\=\,]\mu_k,\nu_*[$ in $\cV_1$ has
  length $(b_*b_k)^{-1}$ in the skewness metric,
  where $b_*=-\ord_{E_*}(L)$, $b_k=-\ord_{p_k}(L)$ for
  a generic affine function $L$ on $\C^2$.
    
  For large positive integers $m_0,m_1,\dots,m_l=m_0$ 
  to be determined shortly, define valuations 
  $\nu_k$, $0\le k<l$ as follows: blow up 
  $p_k$, then successively $m_k$ times
  blow up the intersection point between the 
  (strict transform of) $E_*$ and the previously
  obtained exceptional divisor.
  The segment $I_k\=\,]\nu_k,\nu_*[\,\subset J_k$ 
  then has length $(b_*(b_k+m_kb_*))^{-1}$ in the
  skewness metric.

  For $m_k$ large, the segment $I_k$ is small enough
  so that $F_\bullet$ maps $I_k$ homeomorphically onto 
  a subsegment of $J_{k+1}$. Moreover, when the segments 
  are parameterized by skewness, $F_\bullet$ 
  is given by a M{\"o}bius map with nonnegative integer
  coefficients, see~\cite[Theorem~7.4]{eigenval}.
  Thus the one-sided derivative of
  $F_\bullet$ on $I_k$ at $\nu_*$ is a well
  defined rational number $s_k>0$.

  The key fact is now that the M{\"o}bius property above 
  implies that the iterate $F^l_\bullet$ 
  maps the segment $I_0$ into itself and that either
  $F^l_\bullet\equiv\id$ or $F^l_\bullet$ is a contraction
  on $I_0$, see~\cite[Lemma~5.5]{eigenval}. 
  The former case is impossible by Theorem~\ref{thm:basin}.
  Hence we conclude that $\prod_{k=0}^{l-1}s_k<1$.

  Pick $\e>0$ such that $\prod_{k=0}^{l-1}s_k\le(1-2\e)^l$.
  We may then pick the integers $m_k$ (with $m_0 = m_l$) 
  above arbitrarily large so that 
  \begin{equation}\label{e3}
    \frac{s_k}{m_kb_*+b_k}\le(1-\e)\frac{1}{m_{k+1}b_*+b_{k+1}}
    \quad\text{for $0\le k<l$};
  \end{equation}
  we just need to make $m_{k+1}/m_k$ slightly smaller
  than $1/s_k$.
  By the definition of $s_k$,~\eqref{e3} implies that
  $F_\bullet$ maps $I_k$ into $I_{k+1}$.
  Let $U_k$ be the open subset of $\cV_0$ consisting of
  valuations whose tree retraction to the closed 
  segment $\overline{I_k}$ is contained in $I_k$. 
  Then $F_\bullet$ maps $U_k$ into $U_{k+1}$ for $0\le k<l$, 
  assuming the $m_k$'s are large enough 
  (again, by convention, $U_l=U_0$).

  Let $X$ be the smallest admissible compactification 
  of $\C^2$ dominating $X_0$ such that the center of 
  $\nu_k$ is one-dimensional for $0\le k<l$. 
  Then $X$ is obtained from $X_0$ by performing all
  the blowups mentioned above, so $X$ is tight 
  and the morphism $X\to X_0$ induced by the identity on 
  $\C^2$ is an isomorphism above $X_0\setminus Z$.
  
  The center of $\nu_k$ on $X$ intersects 
  (the strict transform of) $E_*$ at some point 
  $q_k$ and the open set $U_k$ above exactly consists 
  of the valuations in $\cV_0$ whose center on 
  $X$ equals $q_k$.
  Thus $F_\bullet U_k\subset U_{k+1}$ translates into 
  the lift $\tF:X\dashrightarrow X$ of $F$ 
  being holomorphic at~$q_k$.

  This completes the proof when $Z$ consists of
  a single periodic orbit. In general, there are 
  several orbits, but we can handle them one at a time.
\end{proof}
\begin{proof}[Proof of Corollary~\ref{cor:divnonendrec}]
  There is a finite subset $Z\subset E_*$ such that if
  $p\in E_*\setminus Z$,
  we have $d(F,\cdot)\equiv\la_1$ on the 
  open set $U(p)\subset\cV_0$ of valuations whose center
  on $X$ is $p$. We may pick $l\ge1$ such that any periodic
  orbit of $\tF|{E_*}$ intersecting $Z$ has order
  dividing $l$.

  As in the proof of Corollary~\ref{cor:nondivrec} 
  we may assume $F^n_\bullet\to\nu_*$ as $n\to\infty$.
  Similarly, we may assume $F^n_\bullet\nu\ne\nu_*$ 
  for all $n$.

  Pick $n=n(\nu)$ so that the center of 
  $F^k_\bullet\nu$ on $X$ is a point 
  $p_k\in E_*$ for $k\ge n$.
  After increasing $n$ we may assume that the 
  orbit $p_n, p_{n+1},\dots$ is either disjoint from $Z$,
  or periodic of order (dividing) $l$.
  In the first case, 
  $d(F^{k+j+1},\nu)=\la_1d(F^{k+j},\nu)$ for any $j\ge0$.
  In the second case we conclude the
  proof as in Corollary~\ref{cor:nondivrec}.
\end{proof}
%
%
%%%%%%%%%%%%%%%%%%%%%%%%%%%%%%%%%%%%%%%%%%%%%%%%%%%%%%%%%%%%%%%%%%%
%
%
\section{Maximum topological degree: $\la_2= \la_1^2$}\label{maximum}
Next we turn to maps with maximum topological 
degree $\la_2=\la_1^2$. As we do not have an analog of Theorem~\ref{thm:spectral} at our disposal,
we base our analysis on a 
detailed description of the dynamics of $F_\bullet$ on
$\cV_1$ using tree arguments.
%
%%%%%%%%%%%%%%%%%%%%%%%%%%%%%%%%%%%%%%%%%%%%%%%%%%%%%%%%%%%%%%%%%%%
%
\subsection{Dynamics on $\cV_1$}\label{s:V1dyn}
The results of this section form the basis for the proof of 
Theorems~A,~B and~C in the case $\la_2 = \la_1^2$. 
Define
\begin{equation}\label{e:TF}
  \cT_F\=\{\nu\in\cV_1\ |\ F_\bullet\nu=\nu\}.
\end{equation}
This set is nonempty by Proposition~\ref{prop:eigenval}.
The following three results summarize the structure of $\cT_F$
and its dynamical significance.
\begin{prop}\label{p:nonproper}
  Suppose $\deg F^n/\la_1^n$ is unbounded. 
  Then $\cT_F=\{\nu_*\}$ is
  a singleton, where $\nu_*$ is a rational pencil valuation,
  and $F^n_\bullet\to\nu_*$ on $\cV_1$ as $n\to\infty$.
  Moreover, $F$ is not proper, $\deg F^n\sim n\la_1^n$ and
  there exist affine coordinates in which 
  $F(x,y)=(P(x),A(x)y^{\la_1}+O_x(y^{\la_1-1}))$, where 
  $\deg P=\la_1$ and $\deg A\ge1$.
\end{prop}
\begin{prop}\label{p:boundeddeg}
  Suppose $\deg F^n$ is bounded.
  Then $F$ is a polynomial automorphism of $\C^2$.
  In suitable affine coordinates, either
  \begin{itemize}
  \item[(a)]
   $F$ is an affine map and $-\deg\in\cT_F$; or
  \item[(b)]
    $F$ is a skew product of the form 
    $F(x)=(ax+b,cy+P(x))$, where $a,c\in\C^*$, $b\in \C$;
    we may then assume $\cT_F=[\nu_0,\nu_1]$, 
    where $\nu_1$ is associated to the pencil $x=\mathrm{const}$
    and $\nu_0$ is a monomial valuation satisfying
    $\nu_0(y)=-1$, $\nu_0(x)=-1/q$, where $q=\deg P>1$.
  \end{itemize}
\end{prop}
\begin{prop}\label{p:proper}
  Suppose $\deg F^n/\la_1^n$ is bounded and $\la_1>1$. 
  Then $F$ is proper. Moreover:
  \begin{itemize}
  \item[(a)]
    either $\cT_F$ consists of a single quasimonomial
    valuation $\nu_*\in\cV_1$ with $\a(\nu_*)>0$;
    or $\cT_F$ is a closed segment in $\cV_1$
    whose endpoints are divisorial valuations;
  \item[(b)]
    $\cT_{F^n}=\cT_{F^2}$ for $n\ge2$, 
    and either $\cT_F=\cT_{F^2}$ or $\cT_F$ is a singleton, 
    lying in the interior of $\cT_{F^2}$;
  \item[(c)]
    for $\nu\in\cV_1$, $F^{2n}_\bullet\nu\to r(\nu)$ as $n\to\infty$,
    where $r:\cV_1\to\cT_{F^2}$ is the natural retraction;
  \item[(d)]
    in suitable affine coordinates, all the 
    valuations in $\cT_{F^2}$ are monomial.
  \end{itemize}
\end{prop}
The convergence in (c) holds in a 
strong sense:  $F^{2n}_\bullet\nu\to r(\nu)$ weakly and
$A(F^{2n}_\bullet\nu)\to A(r(\nu))$. 

The proofs of these results are given in Section~\ref{sec:treeproofs}.
%
%%%%%%%%%%%%%%%%%%%%%%%%%%%%%%%%%%%%%%%%%%%%%%%%%%%%%%%%%%%%%%%%%%%
%
\subsection{Proof of Theorem~C}
If $\deg(F^n)/\la_1^n$ is unbounded, then we are in case (1) 
by Proposition~\ref{p:nonproper}.

If $\deg(F^n)$ is bounded, then we pick suitable affine
coordinates as in Proposition~\ref{p:boundeddeg}. 
When $F$ is affine, it extends holomorphically to $X=\P^2$.
When $F$ is a skew product as in~(b), it extends
holomorphically to the Hirzebruch surface $X=\F_q$.
Indeed, we can view $\F_q$ as a toric surface
associated to the complete fan generated by the vectors
$(1,0)$, $(0,1)$, $(0,-1)$ and $(-q,-1)$ in $\R^2$,
see~\cite[pp.6--8]{fulton}.
Then $X$ is a compactification of $\C^2$ and
$X\setminus\C^2$ is a union of two rational curves, 
corresponding to the centers of $\nu_0$ and $\nu_1$ on $X$.
As $\nu_0$ and $\nu_1$ are totally invariant under $F_\bullet$, 
the lift of $F$ to $\F_q$ is holomorphic. 

Finally, when $\deg(F^n)/\la_1^n$ is bounded but $\la_1>1$,
we apply Proposition~\ref{p:proper}.
Hence we may assume that $\cT_{F^2}$ consists of monomial 
valuations. If $\cT_F$ contains a divisorial 
valuation $\nu_*$, then 
$\nu_*(x)=-p/q$, $\nu_*(y)=-1$, where $q\ge p\ge1$ 
and $\gcd(p,q)=1$. 
Let $X\= X_{p,q}$ be the toric surface
associated to the complete fan generated by the vectors
$(1,0)$, $(0,1)$ and $(-p,-q)$ in $\R^2$.  
Then $X$ has at worst quotient singularities and is 
in fact a weighted projective plane, see below.
Note that
$X\setminus\C^2$ is a single, totally invariant, 
rational curve.
As $\nu_*$ is totally invariant under $F_\bullet$,
$F$ lifts to a holomorphic selfmap of $X$.

If $\cT_F$ contains no divisorial valuation, then 
$\cT_F=\{\nu_*\}$, where $\nu_*$ is an irrational 
valuation belonging to the interior of $\cT_{F^2}$.
Note that $d(F,\cdot)$ cannot be locally constant
at $\nu_*$, or else the tangent map of $F$ at $\nu_*$
(see below)
would be the identity and $\cT_F=\cT_{F^2}$.
Thus $\la_1=d(F,\nu_*)$ is irrational.

Pick any divisorial valuation $\nu\in\cT_{F^2}$ with $\nu<\nu_*$,
and set $\nu'=F_\bullet\nu$. 
Then $\nu'>\nu_*$ and $\nu$, $\nu'$ are both totally invariant 
under $F^2_\bullet$.
We have $\nu(y)=\nu'(y)=-1$, $\nu(x)=-p/q$, $\nu'(x)=-p'/q'$ for some
integers with $q\ge p\ge1$, $q'>p'\ge1$ and
$\gcd(p,q)=\gcd(p',q')=1$.
Define $X$ to be the toric surface associated to the complete
fan generated by $(1,0)$, $(0,1)$, $(-p,-q)$, $(-p',-q')$ in $\R^2$.
Then $X$ has at worst quotient singularities, and
$X\setminus \C^2$ consists of two irreducible 
rational curves $E,E'$ corresponding to the centers 
of $\nu$ and $\nu'$ on $X$. As these valuations are permuted 
by $F_\bullet$ and 
totally invariant by $F^2_\bullet$, $F$ lifts to a 
holomorphic selfmap of $X$ which permutes $E$ and $E'$.
This completes the proof of Theorem~C.
%
%%%%%%%%%%%%%%%%%%%%%%%%%%%%%%%%%%%%%%%%%%%%%%%%%%%%%%%%%%%%%%%%%%%
%
\subsection{Weighted projective spaces and normal forms}\label{subsecl12l2}
Suppose $\la_2=\la_1^2$, $\deg F^n/\la_1^n$ is bounded and $\la_1>1$.

Assume that $\cT_F$ contains a divisorial valuation 
$\nu_*$, with $\nu_*(x)=-p/q$,  $\nu_*(y)=-1$, where $q\ge p\ge1$ 
are relatively prime integers.
We saw that $F$ is holomorphic on the toric surface $X_{p,q}$. 
Conversely any polynomial map of $\C^2$ which extends as a 
holomorphic map to $X_{p,q}$ satisfies  $\la_2=\la_1^2$, 
and $\deg F^n/\la_1^n$ is bounded. 
The surface $X_{p,q}$ is the weighted projective space with 
homogeneous coordinates $[x:y:z] \sim [\la^p x : \la^q y :\la  z]$ 
for all $\la\in\C^*$, see~\cite[p.35]{fulton},~\cite{dolgachev} or~\cite{cox}.
For any polynomial $P$, let $P_+$ be its $\nu_*$-leading 
homogenous part, \ie the sum of all monomials
$a_{ij}x^iy^j$ in $P$ such that $-(pi/q+j)=\nu_*(P)$. Then 
a polynomial map $ F = (P,Q)$ is holomorphic on $X_{p,q}$ 
iff $P_+$ and $Q_+$ have no common zeroes on the 
weighted projective line
$[x:y] \sim [\la^p x : \la^q y]$, \ie iff $P_+(x^q,y^p)$ 
and $Q_+(x^q,y^p)$ have no common zeroes in 
$\C^2\setminus \{ 0\}$.
 
Note that there exist polynomial maps of $\C^2$ which extend 
to a unique $X_{p,q}$. 
For such an example, pick any $\la_1$ divisible by $p$ and $q$, 
write $\la_1 = pa = qb$ and take
$F(x,y)=(P,Q)$ with 
$P_+=\a x^{\la_1}+\b y^{pb}$, 
$Q_+=\g x^{qa}+\d y^{\la_1}$, where $\a\b\g\d \neq 0$. 
When $\la_1$ is not divisible by $\lcm(p,q)$, 
$\cT_F$ is not reduced to a singleton.
 
Pick $p,q$ and $p',q'$ any two pairs of relatively prime integers 
with associated monomial valuations $\nu$ 
and $\nu'$ and $p'/q' > p/q$. Then there exists a polynomial 
map of $\C^2$ for which $\cT_F$ is precisely the segment of 
monomial valuations  $[\nu,\nu']$. Take $\la_1$ divisible by $\lcm(p,q,p',q')$, 
write $\la_1 = pa = qb= p'a'=q'b'$ 
and define 
$F(x,y)=(\a x^{\la_1}+\b y^{pb}+C_0,\d y^{\la_1}+\g x^{q'a'}+ C_1)$  
with $\a\b\g\d \neq 0$, and $C_0,C_1\in \C$.
 
Finally if $\cT_F$ contains no divisorial valuation, then $\la_1\not\in\N$
and $\cT_F$ consists of a single irrational monomial 
valuation $\nu_*$.
We may assume
$\nu_*(x)=-t$, $\nu_*(y)=-1$, where $t\in (0,1)$ is
irrational. This leads to
\begin{equation*}
  P_+=\a y^b
  \quad\text{and}\quad
  Q_+=\b x^c
\end{equation*}
where $b,c\in\N$, $0<b<c$, 
$bc=\la_2=\la_1^2$, $t=\sqrt{b/c}$ 
and $\a,\b\in\C^*$.

%
%%%%%%%%%%%%%%%%%%%%%%%%%%%%%%%%%%%%%%%%%%%%%%%%%%%%%%%%%%%%%%%%%%%
%
\subsection{Proofs of Propositions~\ref{p:nonproper}-\ref{p:proper}}\label{sec:treeproofs}
The arguments utilize the tree structure of $\cV_1$ 
in much more detail than other parts of this paper.
In particular, we need to exploit the relationship between the 
parameterizations $\a$ and $A$ on the tree $\cV_0$ as explained
in~\cite[Appendix~A]{eigenval}. 
There is an increasing, lower semicontinuous 
\emph{multiplicity} function 
$m:\cV_0\to\N^*\cup\{+\infty\}$ 
such that $A(\nu) = -2 -\int_{-\deg}^\nu m(\mu)d\a(\mu)$ 
for all $\nu\in\cV_0$, see~\cite[Theorem~A.4]{eigenval}. 
The multiplicity of any quasimonomial valuation is finite,
whereas infinitely singular valuations have infinite
multiplicity.

Write $JF$ for the Jacobian determinant of $F$. 
The multiplicity function will be primarily exploited through the
following \emph{Jacobian formula}, 
see~\cite[Lemma~7.6]{eigenval}
\begin{equation}\label{e:jacobian}
  A(\nu) + \nu(JF) = d(F,\nu)A(F_\bullet\nu).
\end{equation}
We start by proving some general facts about the set 
$\cT_F$ defined in~\eqref{e:TF}.
\begin{lem}\label{l:TFstruc}
  The set $\cT_F$ is nonempty. 
  For every $\nu\in\cT_F$, $F_*\nu=\la_1\nu$ and
  $F^*Z_\nu=F_*Z_\nu=\la_1Z_\nu$. 
  If $F$ is proper, every $\nu\in\cT_F$ is 
  totally invariant under $F_\bullet$.
\end{lem}
\begin{proof}
  By Proposition~\ref{prop:eigenval} there exists $\nu\in\cV_1$
  with $F_*\nu=\la_1\nu$, hence $\cT_F$ is nonempty. For any such 
  $\nu$, we have $F_*Z_\nu=\la_1Z_\nu$ by Lemma~\ref{lem:pushpush}.
  The condition $\la_2=\la_1^2$ and the Hodge Index Theorem then 
  imply $F^*Z_\nu=\la_1Z_\nu$. When $F$ is proper, 
  the latter equation 
  implies that $\nu$ is totally invariant by
  Proposition~\ref{prop:pull}.

  Now pick any $\mu\in\cT_F$. We must prove that $F_*\mu=\la_1\mu$.
  In any case, $F_*\mu=\la\mu$ for some $\la>0$.
  Pick $\nu\in\cV_1$ such that $F_*\nu=\la_1\nu$.
  We may assume $\nu\ne\mu$.
  By what precedes and by~\eqref{e:wedge},
  $\la_1\a(\mu\wedge\nu)=(Z_\mu\cdot F^*Z_\nu)
  =(F_*Z_\mu\cdot Z_\nu)=\la\a(\mu\wedge\nu)$.
  Since $\mu\ne\nu$, $\a(\mu\wedge\nu)>0$ and so $\la=\la_1$.
\end{proof}

\begin{proof}[Proof of Proposition~\ref{p:nonproper}]
  Every valuation $\nu\in\cT_F$ must have $\a(\nu)=0$,
  or else $\deg(F^n)/\la_1^n$ would be bounded by $1/\a(\nu)$.
  By Theorem~B' and~\S7.4 in~\cite{eigenval}, 
  there exists a rational pencil valuation $\nu_*\in\cT_F$. 
  Moreover, in suitable affine coordinates, $\nu_*$ 
  corresponds to the pencil $x=\mathrm{const}$ and
  $F$ takes the required form. 
  Hence $\deg F^n\sim n\la_1^n$.

  Now note that since $F^*Z_{\nu_*}=\la_1Z_{\nu_*}$, we have
  \begin{equation}\label{e:skew}
    \a(F^n_\bullet\nu\wedge\nu_*)
    =\frac{\la_1^n}{d(F^n,\nu)}\a(\nu\wedge\nu_*)
  \end{equation}
  for any $\nu\in\cV_1$ and $n\ge1$. 

  Apply this to $n=1$ and $\nu<\nu_*$ close to $\nu_*$. 
  Then $F_\bullet\nu<\nu_*$, so $d(F,\nu)<\la_1$
  or else $\nu\in\cT_F$. 
  Hence $d(F,\cdot)$ is nonconstant near $\nu_*$. 
  By~\cite[Proposition~7.2]{eigenval} there exists 
  $\nu_0\in\cV_0$, $\nu_0>\nu_*$ such that 
  $F_*Z_{\nu_0}=c\cL$ where $c\ge0$ and $\cL$ is the class
  of a line. Then
  $c=(F_*Z_{\nu_0}\cdot Z_{\nu_*})=\la_1\a(\nu_0\wedge\nu_*)=0$.
  In particular, $F$ is not proper, 
  see Proposition~\ref{prop:proper-crit}. 

  It remains to prove that $F^n_\bullet\nu\to\nu_*$ as
  $n\to\infty$ for every $\nu\in\cV_1$. This will in particular imply 
  $\cT_{F^n}=\{\nu_*\}$ for all $n\ge1$.
  When $\a(\nu)>0$, this follows from~\eqref{e:skew},
  since $d(F^n,\nu)\ge\deg(F^n)\a(\nu)$, so
  suppose $\a(\nu)=0$ and set $\nu'=F_\bullet\nu$.
  If $\a(\nu')>0$, then again $F^n_\bullet\nu\to\nu_*$, 
  so suppose $\a(\nu')=0$. 
  Consider the nef Weil class $Z_{\nu'}$ 
  and note that $(F^*Z_{\nu'}\cdot Z_\nu)=0$.
  By the Hodge Index Theorem, $F^*Z_{\nu'}=c Z_\nu$,
  where $c=\la_2/d(F,\nu)>0$. 
  Let $\nu_0>\nu_*$ be the valuation with $F_*Z_{\nu_0}=0$ 
  considered above.
  Then 
  $0=(F_*Z_{\nu_0}\cdot Z_{\nu'})
  =(Z_{\nu_0}\cdot F^*Z_{\nu'})=c\a(\nu_0\wedge\nu')$,
  which implies $\nu'=\nu_*$, 
  completing the proof.
\end{proof}
\begin{proof}[Proof of Proposition~\ref{p:boundeddeg}]
  Note that $\la_1=\la_2=1$.
  It suffices to prove that $F$ is proper. Indeed,
  then $F$ is a polynomial automorphism, and all 
  the assertions follow from the 
  Friedland-Milnor classification~\cite{FriedlandMilnor}.

  Now, if $F$ were not proper, 
  by Proposition~\ref{prop:proper-crit} we could find 
  a divisorial valuation $\nu_0\in\cV_0$ such that
  $F_*Z_{\nu_0}=0$. Hence, for any $\nu\in\cT_F$, 
  $0=(F_*Z_{\nu_0}\cdot Z_\nu)=\la_1\a(\nu_0\wedge\nu)$,
  so that $\a(\nu)=0$ and $\nu_0\ge\nu$. Thus $F$ would be of the form
  $(ax+b, C(x) y + D(x))$ in suitable coordinates. 
  As $F$ is nonproper, 
  $\deg C \ge 1$, contradicting that $\deg F^n$ is bounded.
\end{proof}

\medskip
Next we turn to Proposition~\ref{p:proper}, 
which is significantly harder to prove than the previous 
two propositions.
Assume therefore, for the rest of Section~\ref{sec:treeproofs} 
that $\la_2=\la_1^2>1$ and that $\deg F^n/\la_1^n$ is bounded. 
Then $F$ is proper as follows from the proof of Proposition~\ref{p:boundeddeg}.
Hence, by Lemma~\ref{l:TFstruc},
every $\nu\in\cT_F$ is totally invariant under $F_\bullet$.

To continue the proof, we 
recall the definition of the (tree) tangent map of $F$ at
any valuation $\nu\in\cV_0$, see~\cite[\S3]{eigenval}. 
Declare two segments of the form 
$[\nu_1,\nu[\,$ and $[\nu_2,\nu[\,$ to be equivalent iff
they have nonempty intersection. An equivalence class is called 
a \emph{tangent vector} at $\nu$ and the set $T\nu$ 
of tangent vectors the \emph{tangent space} at $\nu$. 
If $\vv$ is a tangent vector, we denote by $U(\vv)$ the open set of all
valuations determining $\vv$. These open sets form a basis for the weak topology 
on $\cV_0$~\cite[Theorem~5.1]{valtree}.
As $F_\bullet$ preserves the tree structure it naturally induces a 
surjective selfmap $\mathsf{F}:T\nu\to T\nu$, the \emph{tangent map},
for any eigenvaluation $\nu\in\cT_F$.

When $\nu\in\cT_F$ is infinitely singular, $T\nu$ is a singleton, so
$\mathsf{F}\equiv\id$. 
When $\nu\in\cT_F$ is irrational, $T\nu$ consists of two 
tangent vectors and $\mathsf{F}^2\equiv\id$.
If instead $\nu\in\cT_F$ is divisorial and 
$X$ is an admissible compactification for which
the center of $\nu$ is a prime $E$ of $X$, then there exists 
a canonical identification  of $E$ with the tangent space $T\nu$
at $\nu$ as follows. 
For any point $p\in E$, all valuations centered at $p$ 
determine the same tangent vector $\vv_p$ at $\nu$.
Conversely all valuations in $U(\vv)$ are centered along a connected subspace
intersecting $E$ in a single point $p(\vv)$, see~\cite[Theorem~B.1]{valtree}.
With this identification, $\mathsf{F}$ can be viewed as a 
rational selfmap of $E\simeq\P^1$.
\begin{lem}\label{l:tanvec}
  Assume $\nu_*\in\cT_F$ and consider a tangent vector $\vv$ 
  at $\nu_*$ represented by a valuation with $\a>0$. 
  If $\vv$ is totally invariant by $\mathsf{F}$, then 
  $F_\bullet\equiv\id$  on a small segment representing $\vv$.
\end{lem}
\begin{lem}\label{l:tanvec0}
  Assume $\nu_*\in\cT_F$ is divisorial but not a rational pencil
  valuation 
  Then $F_\bullet U(\vv)=U(\mathsf{F} \vv)$ for any tangent 
  vector $\vv$ at $\nu_*$.
\end{lem}
With these two lemmas at hand, we continue the proof of 
Proposition~\ref{p:proper}.

First suppose $\cT_F=\{\nu_*\}$ is a singleton.
Then $\nu_*$ cannot be infinitely singular 
by Lemma~\ref{l:tanvec}. Neither can it be a rational pencil
valuation, since then $\deg F^n\sim n\la_1^n$. 
Hence $\nu_*$ is quasimonomial with $\a(\nu_*)>0$.

If $\cT_F$ is not a singleton, pick two distinct 
valuations $\nu_1,\nu_2 \in \cT_F$. 
The set $[\nu_1,\nu_2]\cap\cT_F$ is clearly closed, 
and Lemmas~\ref{l:tanvec} and~\ref{l:tanvec0} show that it is open. 
Hence $[\nu_1,\nu_2] \subset \cT_F$, so 
$\cT_F$ is a subtree of $\cV_1$. 
Suppose $\nu$ is a branch point of $\cT_F$. Then $\nu$ is divisorial and $\a(\nu)>0$. 
Each branch of $\cT_F$ emanating from $\nu$ corresponds 
to a tangent vector $\vv$ which is totally invariant 
by the tangent map $\mathsf{F}$ at $\nu$, 
since a valuation in $\cT_F$ is totally invariant.
Now $\mathsf{F}$ can be identified with a rational map on $\P^1$ 
whose degree equals $\la_2/d(F,\nu)=\la_1 >1$.
Hence $\mathsf{F}$ admits 
at most two totally invariant tangent vectors. 
This gives a contradiction.
We have shown that $\cT_F$ is a non-empty closed segment of $\cV_1$.

Suppose $\cT_F\subsetneq\cT_{F^2}$ and that $\cT_F$ is 
not a singleton. We can then find a valuation $\nu$ which is 
an interior point of $\cT_{F^2}$ but an endpoint of $\cT_F$.
Consider the tangent map $\mathsf{F}$ at $\nu$. We see
that $\mathsf{F}^2$ admits two totally invariant tangent vectors,
exactly one of which is (totally) invariant by $\mathsf{F}$.
This is a contradiction. Similarly, $\cT_F$ cannot be a 
singleton consisting of an endpoint of $\cT_{F^2}$.
An analogous argument shows that 
$\cT_{F^n}=\cT_F$ for all $n\ge3$, establishing~(b).

Now suppose that $\cT_F$ is a nontrivial segment, and 
consider a subsegment $I=]\nu_1,\nu_2[\subset\cT_F$,
that is totally ordered, \ie $\nu_1<\nu_2$. We claim that
the multiplicity function $m$ is constant on $I$. 
This will imply that the endpoints of $\cT_F$ are 
divisorial (rather than infinitely singular), and hence
complete the proof of~(a).
The Jacobian formula~\eqref{e:jacobian} gives 
$\nu(JF)=(\la_1-1)A(\nu)$, so the function 
$\nu\to\nu(JF)$ is piecewise affine on $I$ 
(with respect to skewness) with slope $m(\nu)(\la_1-1)$.
Now $\nu\to\nu(JF)$ is concave on $I$~\cite[\S A.4]{eigenval}, 
whereas $\nu\mapsto m(\nu)$ is nondecreasing on $I$. 
Thus $m$ is constant on $I$, as required.

Next we turn to~(c).
We may replace $F$ by $F^2$ so that $\cT_F=\cT_{F^2}$.
Pick $\nu\in\cV_1\setminus\cT_F$ 
and write $\nu_*\=r(\nu)\in \cT_F$.
Then $\nu_*$ is divisorial and $\a(\nu_*)>0$. 
We need to show that $F^n_\bullet\nu\to \nu_*$ as $n\to\infty$.
Denote by $\vv$ the tangent vector at $\nu_*$ represented by $\nu$. 
Note that $U(\vv)\cap\cT_F=\emptyset$.
If $\vv$ is not preperiodic, then 
for $n$ large enough, the functions 
$\mu \mapsto \mu(JF)$ and $\mu \mapsto d(F,\mu)$ are both 
constant on $ U(\vv_n)$,
where $\vv_n \= \mathsf{F}^n\vv$, see~\cite[Proposition~3.4]{eigenval}.
By Lemma~\ref{l:tanvec0}, $F^n_\bullet \nu \in U(\vv_n)$ for all $n$,
and the Jacobian formula~\eqref{e:jacobian} implies 
$|A(F^n_\bullet \nu)-A(\nu_*)|\sim\la_1^{-n} \to 0$.
This implies $F^n_\bullet\nu\to\nu_*$ in the weak topology.

When $\vv$ is preperiodic, we may assume it is fixed but 
not totally invariant by $\mathsf{F}$.
Let $I \= [\nu_*,\nu]$ and 
$\Omega \= \{ \mu \in I\ |\ F^n_\bullet\mu\to\nu_*\}$. 
As in the proof of Lemma~\ref{lem:quotient}, 
$F_\bullet$ is given  on $I$ by a (piecewise) M\"obius transformation 
with non-negative integer coefficient fixing $\nu_*$. 
Hence $\Omega$ contains a small neighborhood of $\nu_*$ in $I$, and is therefore open. 
\begin{lem}\label{l:equicont}
  Let $\nu_1,\nu_2\in\cV_1$ be comparable 
  (\ie $\nu_1\le\nu_2$ or $\nu_2\le\nu_1$)
  valuations with $\a(\nu_i)>0$, $i=1,2$.
  Then
  \begin{equation}\label{e:equicont}
    |A(F_\bullet\nu_2)-A(F_\bullet\nu_1)|
    \le\frac{4}{\a(\nu_1)\a(\nu_2)}|A(\nu_2)-A(\nu_1)|
  \end{equation}
  for any dominant polynomial mapping $F:\C^2\to\C^2$.
\end{lem}
This result implies the (local) equicontinuity of the family 
$(F^n_\bullet)$ on $\Omega \cap \{ \a >0 \}$, hence
$\Omega \cap \{ \a >0 \}$ is closed. 
When $\a(\nu) >0$, we conclude that $\nu\in \Omega$.
Otherwise, $\a(\nu)=0$. To complete the proof, we only
need to show that $\a(\nu_n)>0$ for some $n\ge1$,
where $\nu_n=F^n_\bullet\nu$.

Suppose to the contrary that $\a(\nu_n)=0$
for all $n\ge1$.
As in the proof of Proposition~\ref{p:nonproper} we get
$F^{n*}Z_{\nu_n}=c_nZ_\nu$, where $c_n>0$. 
By Proposition~\ref{prop:pull}, this shows that 
$\nu$ is the only preimage of $\nu_n$ 
under $F^n_\bullet$. 
Let $\vv_n$ be the tangent vector at $\nu_*$ represented by 
$\nu_n$. By Lemma~\ref{l:tanvec0}, $\vv$ is the only 
preimage of $\vv_n$ under $\mathsf{F}^n$. 
If $n\ge2$, this implies that $\vv$ is totally invariant under 
$\mathsf{F}^2$, contradicting Lemma~\ref{l:tanvec} since 
$U(\vv)\cap\cT_{F^2}=\emptyset$.
This proves~(c).

\smallskip
Finally we consider the monomialization statement in~(d).
The starting point is
\begin{lem}\label{lem:moncrit}
  A quasimonomial valuation $\nu\in\cV_1$ is monomial 
  in some affine coordinates iff $A(\nu)+m(\nu)\a(\nu)<0$.
\end{lem}
Define $\nu_*$ as the minimal element in $\cT_{F^2}$.
We claim that, in suitable affine coordinates,
$\nu_*$ becomes a monomial valuation. In view
of Lemma~\ref{lem:moncrit} it suffices to prove
that $A(\nu_*)+m(\nu_*)\a(\nu_*)<0$.
We assume $\nu_*\ne-\deg$. 
By Lemma~\ref{l:tanvec} and the minimality of $\nu_*$,
the tangent vector at $\nu_*$ represented by $-\deg$ is
not totally invariant under the tangent map. 
Hence we can find a small segment 
$I'\subset [-\deg ,\nu_*]$
and another small segment $I$ such that $I\cap I'=\{\nu_*\}$
and $F_\bullet$ maps $I$ homeomorphically onto $I'$.

We use the Jacobian formula~\eqref{e:jacobian}.
If the segments are chosen small enough, we have
$A=-m_*\a+\delta$ on $I'$, where $m_*=m(\nu_*)$ and 
$\delta\in\R$. We must prove that $\delta<0$.
The right hand side of~\eqref{e:jacobian} can be written
\begin{multline*}
  -m_*d(F,\nu)\a(F_\bullet\nu)+\delta\,d(F,\nu) 
  =-m_*(Z_{F_*\nu}\cdot Z_{\nu_*})+\delta\,d(F,\nu) =\\
  =-m_*(Z_\nu\cdot F^*Z_{\nu_*})+\delta\,d(F,\nu)
  =-m_*\la_1\a(\nu_*)+\delta\,d(F,\nu).
\end{multline*} 
As the left hand side in~\eqref{e:jacobian}
is strictly increasing in $\nu$ we get $\delta<0$.
  
Hence $\cT_{F^2}$ contains a monomial valuation.
If $\cT_{F^2}=\{\nu_*\}$ is a singleton, 
then $\cT_F=\{\nu_*\}$, 
$\nu_*$ is divisorial and $\a(\nu_*)>0$,
and there is nothing left to do.
Assume therefore that $\cT_{F^2}$ is a nontrivial segment
that contains at least one monomial valuation.

First suppose that $\cT_{F^2}$ contains a unique 
monomial valuation $\nu_*$; this is then necessarily
divisorial, given by $\nu(x)=-p/q$, $\nu(y)=-1$, where
$q\ge p\ge1$ and $\gcd(p,q)=1$. We wish to change 
coordinates so that $\cT_{F^2}$ contains a nontrivial segment
of monomial valuations. For this, it suffices to prove
that $p=1$. Indeed, there then exists a valuation 
$\nu_1\in\cT_{F^2}$ with $\nu_1>\nu_*$, 
$\nu_1(x)=-1/q$, $\nu_1(y)=-1$ and $\nu_1(y+ax^q)>-1$ for
some $a\in\C^*$. After a change of coordinates 
by a shear of the form $(x,y)\mapsto(x,y+ax^q)$,
$\cT_{F^2}$ will then contain a monomial subsegment.

To prove $p=1$, pick $\nu_1>\nu_*$ such 
that $\nu_1\in\cT_{F^2}$ and such that the multiplicity is
constant, equal to $q$, on $I=\,]\nu_*,\nu_1]$.
The Jacobian formula~\eqref{e:jacobian} 
applied to $F^{2n}$ on $I$ yields
$\nu(JF^{2n})=(\la_1^{2n}-1)A(\nu)$, so 
$\nu\mapsto\nu(JF^{2n})$ is affine on $I$ with slope
$q(\la_1^{2n}-1)$. 
This implies $\deg(JF^{2n})\ge q(\la_1^{2n}-1)$.
Applying the Jacobian formula to $-\deg$ gives
$\deg(JF^{2n})=-\deg(F^{2n})A(F^{2n}_\bullet(-\deg))-2$.
The retraction of $-\deg$ on $\cT_{F^2}$ equals $\nu_*$,
so $A(F^{2n}_\bullet(-\deg))\to A(\nu_*)=-1-p/q$ as $n\to\infty$.
On the other hand, 
$\deg(F^{2n})\le d(F^{2n},\nu_*)/\a(\nu_*)=\la_1^{2n}q/p$.
Altogether this gives $q/p+1\ge q$, so $p=1$.

We may therefore assume that $\cT_{F^2}$ contains a 
nontrivial subsegment consisting of monomial valuations.
We saw that the multiplicity on any totally ordered open subsegment
of $\cT_{F^2}$ is constant, hence $\cT_{F^2}$ is a segment which is 
the union of a segment of monomial valuations $[\nu_0, \nu_1]$, 
$\nu_0 \le \nu_1$ and a totally
ordered segment whose minimum is $\nu_0$.
We may assume 
$\nu_0(x)=-p/q$, $\nu_0(y)=-1$, where $q\ge p\ge1$ and $\gcd(p,q)=1$.
If $p>1$, then the above argument shows that 
$\nu_0$ is an endpoint in $\cT_{F^2}$, yielding
$\cT_{F^2}=[\nu_0,\nu_1]$. 
Hence assume $p=1$ and that $\cT_{F^2}\ne[\nu_0,\nu_1]$.
We proceed by induction on $q$. 
If $q=1$, then $\nu_0=-\deg$ and 
we may change coordinate by an affine map of the form
$(x,y)\mapsto(x,y+ax)$. 
The valuations in $[\nu_0,\nu_1]$ remain monomial, 
but $\cT_{F^2}$ contains a new monomial valuation 
$\nu'_1$, which we may assume maximal. 
Thus $\cT_{F^2}=[\nu'_1,\nu_1]$ and we are done.
If $q>1$, then we change coordinates by a shear of the form
$(x,y)\mapsto(x,y+ax^q)$. 
The valuations in $[\nu_0,\nu_1]$ remain monomial, 
but $\nu_0$ is no longer the minimal monomial valuation in
$\cT_{F^2}$. 
By the inductive hypothesis we may now make $\cT_{F^2}$
monomial. This completes the proof of Proposition~\ref{p:proper}.
%
%%%%%%%%%%%%%%%%%%%%%%%%%%%%%%%%%%%%%%%%%%%%%%%%%%%%%%%%%%%%%%%%%%%
%
\subsection{Proofs of Lemmas}
Finally we prove the lemmas used above.
\begin{proof}[Proof of Lemma~\ref{l:tanvec}]
  Pick a segment $I$ in $\cV_0$ representing the 
  totally invariant tangent vector at $\nu_*$. If $\nu\in I$ 
  is close enough to $\nu_*$, then $\nu'\=F_\bullet\nu\in I$.
  As the tangent vector is totally
  invariant, $\nu$ is then the only preimage of $\nu'$ 
  under $F_\bullet$.
  Write $d=d(F,\nu)$, $\a=\a(\nu)$, $\a'=\a(\nu')$.
  Then $F_*Z_\nu=dZ_{\nu'}$ and since $F$ is proper
  $F^*Z_{\nu'}=\frac{\la_2}{d}Z_\nu$.
  On the one hand, this gives
  \begin{equation*}
    \la_2\a'
    =\la_2(Z_{\nu'}\cdot Z_{\nu'})
    =(F^*Z_{\nu'}\cdot F^*Z_{\nu'})
    =\frac{\la_2^2}{d^2}(Z_\nu\cdot Z_\nu)
    =\frac{\la_2^2}{d^2}\a.
  \end{equation*}
  On the other hand, we get
  \begin{equation*}
    \frac{\la_2}{d}\a(\nu\wedge\nu_*)
    =(F^*Z_{\nu'}\cdot Z_*)
    =(Z_{\nu'}\cdot F_*Z_*)
    =\la_1\a(\nu'\wedge\nu_*).
  \end{equation*}
  Now either $\nu,\nu'>\nu_*$ or $\nu,\nu'<\nu_*$.
  In both cases, we easily deduce, 
  using $\la_2=\la_1^2$ and $\a(\nu_*)>0$, 
  that $\la_1=d$ and $\a=\a'$.
  Hence $\nu'=F_\bullet\nu=\nu$, which completes the proof.
\end{proof}
\begin{proof}[Proof of Lemma~\ref{l:tanvec0}]
  Pick a tight compactification $X_0$ of $\C^2$
  on which the center of 
  $\nu_*$ is a prime $E_0$ of $X$.
  Let $V\subset\Pic(X_0)$ be the subspace spanned by all primes 
  $E\ne E_0$. Then $V$ lies in the orthogonal complement 
  of the class $Z_0$ determined by the condition 
  $Z_0\cdot Z=\ord_{E_0}(Z)$ for all $Z$. 
  But $\a(\nu_*)>0$, hence $Z_0^2>0$,
  see Lemma~\ref{lem:fundsol}.
  Thus the intersection form is negative definite on $V$. 
  We may therefore contract all primes $E\ne E_0$:
  we get a normal (singular) surface $X_1$ 
  on which the center of $\nu_*$ is a prime $E_1$,
  and the lift $\tF:X_1\to X_1$ is holomorphic. 
  Pick a tangent vector $\vv$ at $\nu_*$. 
  It corresponds to a point $p\in E_1$: the set $U(\vv)$ is
  exactly the set of valuations $\nu\in\cV_0$ centered at $p$.
  The map $\tF$ induces a finite germ $(X_1,p)\to(X_1,\tF(p))$ 
  hence $F_\bullet$ maps $U(\vv)$ onto $U(\mathsf{F}(\vv))$.
\end{proof}
\begin{proof}[Proof of Lemma~\ref{l:equicont}]
  The proof is based on the Jacobian formula~\eqref{e:jacobian}.
  Write $\a_i=\a(\nu_i)>0$, 
  $A_i=A(\nu_i)<0$,
  $A'_i=A(F_\bullet\nu_i)<0$,
  $d_i=d(F,\nu_i)>0$ and
  $J_i=\nu_i(JF)<0$
  for $i=1,2$
  and write $d=\deg F$.
  The functions $\nu\mapsto-d(F,\nu)$ and
  $\nu\mapsto\nu(JF)$ define tree potentials on $\cV_0$ 
  in the sense of~\cite[{\S}A.4]{eigenval}. 
  Hence $|d_2-d_1|\le d|\a_2-\a_1|\le d|A_2-A_1|$,
  $|J_2-J_1|\le(\deg JF)|\a_2-\a_1|\le(2d-2)|A_2-A_1|$
  and $d_i\ge d\a_i$.
  Moreover, $|A_i+J_i|\le 2d$.
  Thus~\eqref{e:jacobian} gives
  \begin{multline*}
    |A'_2-A'_1|
    =\left|
      \frac{A_2+J_2}{d_2}-\frac{A_1+J_1}{d_1}
    \right|=\\
    =\left|
      \frac{(A_2+J_2)(d_1-d_2)}{d_1d_2}
      +\frac{(A_2-A_1)+(J_2-J_1)}{d_1}
    \right|
    \le\left(
      \frac{2}{\a_1\a_2}+\frac{2}{\a_1}
    \right)
    |A_2-A_1|,
  \end{multline*}
  which completes the proof since $0<\a_2\le1$.
\end{proof}
\begin{proof}[Proof of Lemma~\ref{lem:moncrit}]
  The proof is based on~\cite[Appendix~A]{eigenval}. By
  the Line Embedding Theorem it is sufficient to prove that
  $A(\nu)+m(\nu)\a(\nu)<0$ iff $\nu\in \cV_1$ is dominated by a
  rational pencil valuation.

  Suppose $C$ is a curve with one place at infinity whose
  associated pencil valuation $\nu_{|C|}$ dominates $\nu$.  Then
  $\a(\nu_{|C|})=0$ and $\a\le0$ on $[\nu,\nu_{|C|}]$, so
  \begin{equation*}
    A(\nu_{|C|})-A(\nu)
    =-\int_{\nu}^{\nu_{|C|}}m(\mu)\,d\a(\mu)
    \ge-\int_{\nu}^{\nu_{|C|}}m(\nu)\,d\a(\mu)
    =m(\nu)\a(\nu).
  \end{equation*}
  By~\cite[Proposition~A.4]{eigenval}, 
  the pencil $|C|$ is rational iff $A(\nu_{|C|})<0$.  
  Hence, $A(\nu)+m(\nu)\a(\nu)<0$ if $|C|$ is rational.
  On the other hand, we may always pick the curve $C$ with
  $\deg(C)=m(\nu)$. Then the multiplicity $m$ is constant equal to
  $m(\nu)$ on $[\nu,\nu_{|C|}]$.
  Thus equality holds above, and we
  get $A(\nu_{|C|})=A(\nu)+m(\nu)\a(\nu)$. 
  So when $A(\nu)+m(\nu)\a(\nu)<0$, $|C|$ is rational.
\end{proof}
%
%
%%%%%%%%%%%%%%%%%%%%%%%%%%%%%%%%%%%%%%%%%%%%%%%%%%%%%%%%%%%%%%%%%%%
%
%
\section{Proofs of Theorems~A and~B}\label{sec:pfAB}
Fix an embedding $\C^2\subset\P^2$ and consider an
arbitrary polynomial mapping $F:\C^2\to\C^2$. 
We start by making a few general remarks.

First, if $X$ is an admissible compactification of $\C^2$
and the lift $\tF:X\dashrightarrow X$ satisfies 
$\tF^{(j+n)*}=\tF^{*j}\tF^{n*}$ on $\Pic(X)$ for $j\ge 1$, then  
$\deg(F^{j+n})=(\tF^{(j+n)*}\cL\cdot\cL)=(\tF^{*j}\tF^{n*}\cL\cdot\cL)$ 
with $\cL\in\Pic(X)$ the (pull-back of the) 
class of a line in $\P^2$. Hence $(\deg(F^j))_{j\ge n}$ 
satisfies the linear recurrence relation
determined by the characteristic polynomial of the 
linear map $\tF^*:\Pic(X)\to\Pic(X)$. 
In the basis given by the primes of $X$, 
$\tF^*$ can be expressed with integer coefficients.
Hence Theorem~B follows from Theorem~A in this case.

Second, if $F:\C^2\to\C^2$ is not dominant, its image is a point 
or a curve. In either case, we pick an admissible compactification 
$X$ of $\C^2$ such that the map $X\to\P^2$ induced by $F$ is
holomorphic. One can then check that the lift $\tF:X\to X$ is
also holomorphic. This establishes Theorems~A and~B in the 
nondominant case.

Third, if $F:X\dashrightarrow Y$, $G:Y\dashrightarrow Z$ 
are dominant rational maps between surfaces, 
and $F^*,G^*$ denotes the action on the respective Picard groups
of these surfaces, then $F^*\circ G^*=(G\circ F)^*$ iff 
no curve in $X$ is contracted by $F$ to an indeterminacy
point of $G$.

Using this, it is not difficult to see that in the case 
$\la_2<\la_1^2$, Theorem~A follows directly from 
Theorems~\ref{thm:stability1},~\ref{thm:stability2} 
and~\ref{thm:stability3}. We obtain Theorem~B
as a consequence, in view of the argument above, but
we have in any case established stronger versions
in Corollaries~\ref{cor:nondivrec},~\ref{cor:divendrec} 
and~\ref{cor:divnonendrec}.

\smallskip
From now on, assume $F$ is dominant and $\la_2=\la_1^2$.
We shall freely use the results in Section~\ref{maximum}.

If $\deg F^n/\la_1^n$ is unbounded, 
then by Proposition~\ref{p:nonproper}
there is a unique rational pencil valuation $\nu_*$ such that 
$F^n_\bullet\nu\to\nu_*$ for every $\nu\in\cV_1$.
We may then proceed exactly 
as in Section~\ref{stab:div}
and prove analogs of Theorem~\ref{thm:stability2} and
Corollary~\ref{cor:divendrec}. 

If instead $\deg F^n/\la_1^n$ is bounded, then we have already 
seen in the proof of Theorem~C that $F$ lifts to a holomorphic
selfmap of a suitable compactification $X$ of $\C^2$, proving 
Theorem~A. However, this compactification need not be smooth or 
dominate the given compactification $\P^2\supset\C^2$,
so Theorem~B does not immediately follow. 

First assume $\la_1=1$ so that $\deg F^n$ is bounded.
Then $F$ is in particular birational. The argument by 
Diller-Favre~\cite[Theorem~0.1]{DF} gives us an admissible 
compactification $X$ of $\C^2$ such that the lift 
$\tF:X\dashrightarrow X$ is algebraically stable.
Thus Theorems~A and~B hold in this case.

Finally assume $\la_1>1$ and $\deg F^n/\la_1^n$ is bounded
and consider the set $\cT_{F^2}$ of eigenvaluations
for $F^2$. We consider three cases. 
In the first case, $\cT_{F^2}=\cT_F=\{\nu_*\}$ is a singleton. 
Then Proposition~\ref{p:proper} shows that 
$F^n_\bullet\nu\to\nu_*$ for any $\nu\in\cV_1$.
We may then proceed exactly 
as in Section~\ref{stab:div}
and prove analogs of Theorem~\ref{thm:stability3} and
Corollary~\ref{cor:divnonendrec}. 
This gives a precise version of Theorem~A (with $X$ a 
tight---hence smooth---compactification of $\P^2$) 
and Theorem~B.

In the second case, $\cT_{F^2}=\cT_F=[\nu_1,\nu_2]$ 
is a segment, where $\nu_1$ and $\nu_2$ are divisorial.
We can then proceed essentially as in Section~\ref{stab:div}.
Namely, let $X_0$ be the minimal admissible compactification
of $\C^2$ such that the centers of $\nu_1$ and $\nu_2$ are 
one-dimensional. We can then make further blowups to arrive
at a tight compactification $X$ such that
for any prime $E$ of $X$ that is the center of a divisorial 
valuation in $\cT_F$, the lift $\tF:X\dashrightarrow X$ is 
holomorphic at any periodic point of $\tF|_E$. 
In view of Proposition~\ref{p:proper}~(c) this shows that
there exists $n\ge1$ such that the following holds
for any prime $E$ of $X$: either $E$ is the center of a valuation 
in $\cT_F$ and then $\tF{E}=E$; or $F^n$ contracts $E$ onto a point
at which all iterates of $X$ are holomorphic.
Thus a precise form of Theorem~A holds. We can also
prove a precise version of Theorem~B as in
Corollary~\ref{cor:divnonendrec} in this setting; in particular,
the sequence $(d(F^j,\nu)_{j\ge n(\nu)})$ satisfies an integral
linear recursion formula for all $\nu\in\cV_1$.

In the third and final case, $\cT_{F^2}=[\nu_1,\nu_2]$ 
is a nontrivial segment and $\cT_F=\{\nu_*\}$ is a singleton.
Again we can prove a precise version of 
Theorem~B as in Corollary~\ref{cor:divnonendrec};
it suffices to apply the result just proved to $F^2$.
\begin{rmk}\label{r:guedjcounter}
  One can check that for $F(x,y)=(y^3,x^2)$ there is no
  toric admissible compactification of $X\supset\C^2$
  such that the lift of $F$ to $X$ satisfies the 
  properties in Theorem~A.
  However, $F$ is holomorphic on $\P^1\times\P^1$. 
  We conjecture that by adding suitable lower degree terms to $F$,
  no smooth compactification of $\C^2$ (admissible or not) will do.
\end{rmk}
% 
%
%%%%%%%%%%%%%%%%%%%%%%%%%%%%%%%%%%%%%%%%%%%%%%%%%%%%%%%%%%%%%%%%%%%%
%
%
\section{Small topological degrees: $\la_2 \le \la_1$}
Here we prove Theorem~D and provide examples of 
maps with $\la_2 = \la_1$.
%
%%%%%%%%%%%%%%%%%%%%%%%%%%%%%%%%%%%%%%%%%%%%%%%%%%%%%%%%%%%%%%%%%%%%%
%
\subsection{Proof of Theorem~D}\label{green} 
Let $F:\C^2\to\C^2$ be
a dominant polynomial mapping with $\la_2<\la_1$. Define 
\begin{equation*}
  G^+(p)\=\limsup_{n\to\infty}\la_1^{-n}\log^+\|F^np\|.
\end{equation*}
We shall see momentarily that the $\limsup$ is
in fact a limit (and ultimately even a locally uniform one), 
but let us first establish
\begin{lem}\label{lem:loggrowth}
  We have $G^+\le C_1\log^+\|\cdot\|+C_2$ on $\C^2$
  for some constants $C_1,C_2>0$.
\end{lem}
\begin{proof}
  By Theorem~B' in~\cite{eigenval} there exists
  a psh function $U$ on $\C^2$ and a constant $C>0$
  such that $U\circ F\le\la_1 U$ on $\C^2$ 
  and such that $C^{-1}\log\|\cdot\|\le U\le C\log\|\cdot\|$ 
  outside a compact subset of $\C^2$. 
  This easily implies the lemma.
\end{proof}
The key step in the proof of Theorem~D is to 
find a sort of filtration for the dynamics, similar to the
one in the case of polynomial automorphisms,
see \eg~\cite{BS1,FriedlandMilnor,HO}.
\begin{lem}\label{lem:filtration}
  For every $\e>0$ there exists an integer $n_0\ge1$,
  a constant $C>0$ and a partition $\C^2=V\cup V^+$
  with $V^+$ open, $FV^+\subset V^+$ and such that:
  \begin{itemize}
  \item[(i)]
    the $\limsup$ defining $G^+$ is a locally uniform limit
    on $V^+$; $G^+$ is pluriharmonic and strictly positive there;
  \item[(ii)]
    For $p\in V$ we have
    $\log^+\|F^{n_0}p\|\le(\la_2+\e)^{n_0}\log^+\|p\|+C$.
  \end{itemize}
\end{lem}
\begin{proof}[Proof of Theorem~D]
  Apply Lemma~\ref{lem:filtration} with $0<\e<\la_1-\la_2$. 
  Set $U^+=\bigcup_{n\ge0}F^{-n}V^+$ and $K^+\=\C^2\setminus U^+$.
  Then $U^+$ is open and $G^+$ is pluriharmonic and strictly
  positive there. On the other hand, the estimate in~(ii)
  implies $G^+\equiv0$ on $K^+$. Thus $G^+$ is everywhere defined
  and satisfies $G^+\circ F=\la_1G^+$. 

  Let $G^{+*}$ be the upper semicontinuous 
  regularization of $G^+$.
  We shall prove that $G^{+*}=0$ on $K^+$.
  Lemma~\ref{lem:loggrowth} implies
  $G^{+*}\le C_1\log^+\|\cdot\|+C_2$ on $\C^2$ 
  for some constants $C_1,C_2>0$.
  Now pick $p\in K^+$. Then $F^{kn_0}p\in V$ for all 
  $k\ge0$ so using the estimate 
  in~Lemma~\ref{lem:filtration}~(ii) we get
  \begin{equation}\label{e1}
    G^{+*}\circ F^{kn_0}(p)\le C_3(\la_2+\e)^{kn_0}(\log^+\|p\|+C_4)
  \end{equation}
  for suitable constants $C_3,C_4>0$ and all $k\ge1$.
  Now the equality $G^+\circ F=\la_1G^+$ 
  yields the inequality $G^{+*}\circ F\ge\la_1G^{+*}$ 
  (we have equality outside the curves contracted by $F$),
  so~\eqref{e1} gives
  \begin{equation*}
    G^{+*}(p)
    \le C_3\left(\frac{\la_2+\e}{\la_1}\right)^{kn_0}(\log^+\|p\|+C_4)
    \to0\ \text{as $k\to\infty$}.
  \end{equation*}
  Hence $G^{+*}\equiv G^+\equiv0$ on $K^+$. 
  We may now argue as in the proof 
  of~\cite[Proposition~3.4]{BS1} to prove 
  that 
  $G^+$ is continuous and psh on $\C^2$, that the limit
  defining $G^+$ is locally uniform on $\C^2$, 
  and that the support of the 
  current $dd^cG^+$ equals $\partial K^+$.

  The estimate in Theorem~D follows from the corresponding estimate
  in Lemma~\ref{lem:filtration}~(ii). This completes the proof.
\end{proof}
The filtration in Lemma~\ref{lem:filtration} is constructed using
a well chosen compactification of $\C^2$. 
\begin{lem}\label{lem:compactification}
  For every $\e>0$ there exists an integer $n_0\ge1$,
  an admissible compactification $X$ of $\C^2$ and a 
  decomposition $X\setminus\C^2=Z^+\cup Z^-$ into 
  (reducible) curves $Z^+$, $Z^-$ without common
  components such that:
  \begin{itemize}
  \item[(i)]
    if $E$ is any irreducible component of $Z^-$ and $L$ is
    a generic affine function on $\C^2$ then
    \begin{equation}\label{e2}
      \ord_E(F^{n_0*}L)\le(\la_2+\e)^{n_0}\ord_E(L);
    \end{equation}
  \item[(ii)]
    there exists a point $p\in Z^+\setminus Z^-$
    such that $\tF^{n_0}$ is holomorphic in a 
    neighborhood of $Z^+$, 
    $\tF^{n_0}(Z^+)=\{p\}$, 
    $\tF$ is holomorphic at $p$, $\tF(p)=p$;
    and there exist local coordinates 
    $(z,w)$ at $p$ in which $\tF$ takes a simple normal form 
    as in Theorem~\ref{thm:stability1}:
    \begin{itemize}
    \item[(a)]
      if $Z^+$ is locally reducible at $p$, then 
      $Z^+=\{zw=0\}$ and
      $\tF(z,w)=(z^aw^b,z^cw^d)$, where $a,b,c,d\in\N$
      and the 
      $2\times2$ matrix $M$ with entries $a,b,c,d$
      has spectral radius $\la_1$;
    \item[(b)]
      if $Z^+$ is locally irreducible at $p$, then 
      $Z^+=\{z=0\}$ and
      $\tF(z,w)=(z^{\la_1},\mu z^cw+P(z))$,
      where $c\ge1$, $\mu\in\C^*$, and $P$ is a 
      nonconstant polynomial with $P(0)=0$.
    \end{itemize}
  \end{itemize}
\end{lem}
The admissible compactification $X$ will not be tight in general.
We first show how to deduce Lemma~\ref{lem:filtration} from
Lemma~\ref{lem:compactification}, then prove the latter lemma.
\begin{proof}[Proof of Lemma~\ref{lem:filtration}]
  Apply Lemma~\ref{lem:compactification} after having decreased $\e$
  slightly.  We can find a small neighborhood $\Omega_p$ of $p$ of the
  form $\{|z|<\delta_1,|w|<\delta_2\}$ such that
  $\tF\Omega_p\subset\Omega_p$.  Set $V_p=\Omega_p\cap\C^2$,
  $V^+\=\tF^{-n_0}(\Omega_p)\cap\C^2=F^{-n_0}V_p$ and $V=\C^2\setminus
  V^+$.  Then $V^+\subset\C^2$ is open and $FV^+\subset V^+$.

  In local coordinates $(z,w)$ 
  near a prime $E=\{z=0\}$ of $X$,
  the function $\log^+\|\cdot\|$ in $\C^2$ 
  equals $-\ord_E(L)\log|z|+O(1)$.
  Similarly, in local coordinates $(z,w)$ at
  the intersection point between two primes 
  $E=\{z=0\}$ and $E'=\{w=0\}$ we have 
  $\log^+\|(z,w)\|=-\ord_E(L)\log|z|-\ord_{E'}(L)\log|w|+O(1)$.

  Note that $\tF^{-n_0}(\Omega_p)$ contains a neighborhood of 
  $Z^+$ in $X$. Estimate~(ii) in Lemma~\ref{lem:filtration}
  is therefore a consequence of~\eqref{e2}.

  It remains to prove~(i).
  For this we use the normal forms in
  Lemma~\ref{lem:compactification}.
  Suppose we are in case~(a). Then
  $\log^+\|(z,w)\|=-s\log|z|-t\log|w|+\varphi(z,w)$
  in $V_p$ for some constants $s,t>0$ and a 
  bounded function $\varphi$.
  It then follows easily that $\la_1^{-n}\log^+\|F^n\|$ 
  converges locally
  uniformly on $V_p$ to $G^+=-s'\log|z|-t'\log|w|$,
  where $s',t'>0$,
  (the vector $(s',t')$ is proportional to the
  eigenvector with eigenvalue $\la_1$ of $M^t$.)
  Hence $G^+$ is pluriharmonic and strictly positive in $V_p$.
  Since $F^{n_0}V^+\subset V_p$, the same properties
  must hold in $V^+$. This completes the proof
  in case~(a). Case~(b) is similar and left to the reader.
\end{proof}
\begin{proof}[Proof of Lemma~\ref{lem:compactification}]
  We apply Theorem~\ref{thm:stability1}. 
  We shall only treat case~(a) of that theorem, case~(b) being 
  similar.

  Thus we have an admissible (tight) compactification
  $X_0$ of $\C^2$, two primes $E_1$, $E_2$ of $X_0$, 
  intersecting in a point $p$,
  such that the lift $\tF_0:X_0\dashrightarrow X_0$ 
  defines an attracting holomorphic fixed point germ at $p$
  given by 
  $\tF_0(z,w)=(z^aw^b,z^cw^d)$ in suitable local
  coordinates $(z,w)$. 
  Write $Z_0=E_1\cup E_2$. 
  Set $\gamma\=\max_{i=1,2}\ord_{E_i}(L)$, where $L$
  is a generic affine function on $\C^2$.
  Pick $n_0\ge 1$ large enough so that 
  $\gamma\la_2^{n_0}<(\la_2+\e)^{n_0}$. 
 
  By increasing $n_0$ if necessary, we can also assume
  that all valuations in the segment $I=[-\deg,\nu_*]$ 
  are mapped by $F_\bullet^{n_0}$ into the open set 
  $U(p)\subset\cV_0$. Here, as before, $U(p)$
  consists of valuations whose center in $X_0$ is the point $p$. 
  Indeed, $U(p)$ is $F_\bullet$-invariant as $F$ 
  is holomorphic at $p$, 
  and by Theorem~\ref{thm:basin}~(b), 
  $F^n_\bullet \nu$ eventually falls into $U(p)$ 
  for $n$ large enough for any $\nu\in I$.
  
  We can now find an admissible compactification $X$ of $\C^2$
  dominating $X_0$ such that $F^{n_0}$ lifts to a holomorphic
  map $G:X\to X_0$. Since $\tF_0$ was holomorphic at $p$,
  we may assume that the strict transforms of $E_1$ and $E_2$
  in $X$ still intersect in a point $p$. We use the notation
  $E_1$, $E_2$ and $p$ also for these objects on $X$.

  Let $Z^+\subset X\setminus\C^2$ be the connected component 
  of $G^{-1}(p)$ containing $p$.
  Our choice of $n_0$ implies that all primes whose
  associated normalized divisorial valuation in $\cV_0$
  is dominated by the eigenvaluation lie in $Z^+$. 
  In particular the center $L_\infty$ of $-\deg$ in $X$
  belongs to $Z^+$.
  Let $Z^-$ be the union of the primes of $X$ not in $Z^+$.
  Let $\tF:X\dashrightarrow X$ be the rational lift of $F$.
  It is easy to see that~(ii) holds. 
  The key point is to establish~(i).

  The dual graph of $X\setminus\C^2$ being a tree and $Z^+$
  being connected imply that each connected component $W$
  of $Z^-$ contains a unique irreducible component $E=E_W$
  intersecting $Z^+$. Moreover, since $L_\infty\subset Z^+$,
  the normalized divisorial valuations in $\cV_0$ associated
  to the irreducible components of $W$ all dominate
  (in the partial ordering on $\cV_0$) the normalized
  divisorial valuation associated to $E_W$. Thus it
  suffices to verify~(i) for $E=E_W$. 

  Now $G$ must map $E$ onto one of $E_1$ or $E_2$,
  say $G(E)=E_1$. The restriction of $G$ to a
  neighborhood of $E$ has topological degree at most $\la_2^{n_0}$.
  This implies in particular that the coefficient of
  $E$ in the divisor $G^*E_1$ is at most $\la_2^{n_0}$.
  But this coefficient is easily seen to be
  \[
  \frac{\ord_E(F^{n_0*}L)}{\ord_{E_1}(L)}
  \ge \gamma^{-1}\ord_E(F^{n_0*}L)
  \]
  for a generic affine function $L$ on $\C^2$.
  This proves~\eqref{e2} since $\ord_E(L)\ge1$
  and $\gamma\la_2^{n_0}<(\la_2+\e)^{n_0}$.
\end{proof}
%
%
%%%%%%%%%%%%%%%%%%%%%%%%%%%%%%%%%%%%%%%%%%%%%%%%%%%%%%%%%%%%%%%%%%%%%%
%
%
\subsection{Examples with $\la_2 = \la_1$}\label{subsecl1l2}
Very few surface maps with $\la_2=\la_1$ have been described 
in the literature.
We provide here a (presumably incomplete) list of examples.
The mappings in~\eqref{item:DS} below appear 
in~\cite[\S 4]{DSarkiv}.
A classification of quadratic polynomial 
maps with $\la_2= \la_1$ is given in~\cite{guedjquad}.
Note that Proposition~\ref{prop:quadint} implies that 
the eigenvaluation of a map with $\la_1=\la_2>1$ 
is always divisorial or infinitely singular.
\begin{enumerate}
\item
  $F = (A(x), Q(x,y))$ with $A$ affine
  and $\deg_y Q\ge1$. 
  Then $\la_1 = \la_2 = \deg_y Q$,
  and $F$ is proper iff for any fixed $x$, $\deg_y Q(x,y) = \deg_y Q$.
  One can show that any proper polynomial map with 
  $\la_1=\la_2>1$ whose eigenvaluation is divisorial takes
  this form in suitable affine coordinates.
\item 
  $F = (P(x), A(x,y))$ with $\deg_y(A)=1$.
  Then $\la_1=\la_2=\deg(P)$, and $F$ is proper iff 
  $A(x,y)=ay+B(x)$, $a\ne0$.
\item
  $F = (\la xP+a, \mu yP+b)$, or $F = (xP+a, (x+y)P+b)$ 
  with $P= P(x,y)$
  of degree $d-1$, and $a,b \in \C$, $\la,\mu \in \C^*$.  
  Then $\la_1 = \la_2 = d$, the
  eigenvaluation is $-\deg$ which is divisorial, and $F$ is 
  not proper.
\item
  $F = (a y^p+ P(x), x^q)$ with $\deg P = pq$, $ a \in \C$.  Then
  $\la_1 = \la_2 = pq$, $F$ is proper and has an infinitely singular
  eigenvaluation.
\item
  $F = ( x + aP, y + bP)$ with $P = P_d + \mathrm{l.o.t}$, 
  $d = \deg(P) \ge 2$  
  and $P_d(a,b) \ne 0$. 
  Then $F$ is proper and $\la_1 = \la_2 = \deg P$.
\item\label{item:DS}
  $F = (ax + by + c, P(x,y))$ with 
  $\deg P = \deg_y P \ge 2$. 
Then $\la_1 = \la_2 = \deg P$.
\item 
  $F = (P(y), ax+ b + Q(y))$ with $d = \deg P = \deg Q\ge 2$. 
  Then $\la_1 = \la_2 = d$.
\end{enumerate}
There seems to be no general conjecture or approach for studying 
the ergodic properties of these maps.
%
%
%%%%%%%%%%%%%%%%%%%%%%%%%%%%%%%%%%%%%%%%%%%%%%%%%%%%%%%%%%%%%%%%%%%
%
%
\appendix
\section{The Riemann-Zariski space at infinity}\label{RZspace}
In this appendix we briefly develop the necessary material
needed for the proof of Theorem~\ref{thm:basin}.
Most of the discussion is completely analogous to
the one in the paper~\cite{deggrowth}, to which we refer 
for details. See also~\cite{cantat,manin}.

The main new consideration is the construction and study 
of the Weil class $Z_\nu$ associated to a valuation $\nu$ 
centered at infinity.
%
%%%%%%%%%%%%%%%%%%%%%%%%%%%%%%%%%%%%%%%%%%%%%%%%%%%%%%%%%%%%%%%%%%%
%
\subsection{Weil and Cartier classes}
The set of admissible compactifications of $\C^2$
defines an inverse system: $X'$ dominates $X$ if
the birational map $X'\dashrightarrow X$ induced 
by the identity on $\C^2$ is a morphism.
Typically we then identify the primes of $X$
with their strict transforms in $X'$.
Formally, the \emph{Riemann-Zariski space} 
(of $\P^2$ at infinity) is defined as $\fX\=\varprojlim X$. 
(The only difference to~\cite{deggrowth} is that here 
we never blow up points in $\C^2$.) 

Our concern is with classes on $\fX$ rather than $\fX$ itself. 
Given $X$ we let $\hodge(X)$ be 
the vector space of $\R$-divisors on $X$ 
modulo numerical equivalence.
Then $\hodge(X)\simeq\Pic(X)\otimes_\Z\R$.
When $X'$ dominates $X$, the associated
morphism $\mu:X'\to X$ induces linear maps
$\mu_*:\hodge(X')\to\hodge(X)$ and 
$\mu^*:\hodge(X)\to\hodge(X')$ satisfying 
$\mu_*\mu^*=\id$. 
The space of \emph{Weil classes} on $\fX$ is
$W(\fX)\=\varprojlim\hodge(X)$.
We equip it with the projective limit topology.
Concretely, a Weil class $\b\in W(\fX)$ is
given by a collection of classes 
$\b_X\in\hodge(X)$, the \emph{incarnation} of $\b$ on $X$,
compatible under pushforward. 

A class $\b\in\hodge(X)$ for a fixed $X$ defines a Weil class
whose incarnation in any compactification $X'$ dominating $X$
is the pullback of $\b$ to $X'$. Such a Weil class is called 
a \emph{Cartier class}. It is \emph{determined in $X$}.
Formally, the set of Cartier classes
is $C(\fX)\=\varinjlim\hodge(X)$. 

The intersection pairing on each $X$ extends
to a nondegenerate pairing $W(\fX)\times C(\fX)\to\R$.
In particular, we have an inner product on $C(\fX)$.
By the Hodge index theorem, this is of Minkowski type,
allowing us to define the completion $\Ltwo(\fX)$.

A Weil class is \emph{nef} if all its incarnations are nef.
The class of a line in $\P^2$ defines a nef
Cartier class $\cL$ on $\fX$.
The set of nef classes forms a closed convex cone in $W(\fX)$.
Any nef Weil class belongs to $\Ltwo(\fX)$.

%
%%%%%%%%%%%%%%%%%%%%%%%%%%%%%%%%%%%%%%%%%%%%%%%%%%%%%%%%%%%%%%%%%%%
%
\subsection{Classes and valuations}\label{sec:classval}
Every class in $\hodge(X)$ admits a unique representation as
a divisor with support at infinity, \ie a real-valued function 
on the set of primes of~$X$.
Hence a Weil class $Z$ on $\fX$ can be identified 
with a homogeneous function on $\hcVdiv$: its value
at $\nu$ will be denoted $\nu(Z)$.
For example, $\ord_E(\cL)=b_E$, 
where $\cL\in W(\fX)$ is the class of a line on $\P^2$.

Given a valuation $\nu\in\hcV_0$ we define a
Weil divisor $Z_\nu\in W(\fX)$ as follows: 
$\ord_E(Z_\nu)=tb_E\a(\tnu\wedge\nu_E)$ where 
$\a$ denotes skewness and $\nu=t\tnu$, $\tnu\in\cV_0$.
When $\nu$ is divisorial, $Z_\nu$ is Cartier, see 
Lemma~\ref{lem:fundsol} below.
\begin{lem}\label{lem:embedding}
  The assignment $\nu\mapsto Z_\nu$ defines a continuous
  embedding of $\hcV_0$ onto a closed subset of $W(\fX)$.
\end{lem}
\begin{proof}
  After unwinding definitions, the statement boils down to 
  the topology on $\cV_0$ being the weakest topology 
  such that $\nu\mapsto\a(\nu\wedge\nu_E)$ is continuous 
  for all $\nu_E\in\cVdiv$. This in turn follows from the 
  characterization of the topology on $\cV_0$ in terms of the 
  tree structure.
\end{proof}
\begin{lem}\label{lem:fundsol}
  We have $(Z_\nu\cdot W)=\nu(W)$ for $\nu\in\hcVdiv$
  and $W\in C(\fX)$. In particular, for any two valuations 
  $\mu,\nu\in \cV_0$ one has
  \begin{equation}\label{e:wedge}
    (Z_\nu\cdot Z_\mu) = \a(\nu \wedge \mu) \in [-\infty,1].
  \end{equation}
  If $E$ is a prime of some admissible compactification 
  $X$, then $Z_{\ord_E}$ is a Cartier class on $\fX$ determined
  in $X$ by an integral class.
  Thus $b_E^2\a(\nu_E)\in\Z$.
\end{lem}
\begin{proof}
  To prove  $(Z_\nu\cdot W)=\nu(W)$, we pick $X$ 
  such that $W\in C(\fX)$ is determined in $X$. 
  By linearity we may assume $W=E$ is
  a prime of $X$. What we seek to prove 
  is then a special case of a more general formula 
  $(Z\cdot E)=b_E\Delta g\{\nu_E\}$ for any $Z\in\NS(X)$,
  where $g=g_Z$ is the function on $\cV_0$ defined by $Z$
  and $\Delta g$ is its tree Laplacian as defined (up to a sign) 
  in~\cite[Section~A.4]{eigenval}.
  Indeed, $Z=Z_\nu$ as above is chosen so that
  $\Delta g_Z=\delta_{\nu}$.

  Let $E_1,\dots, E_n$ be the primes of $X$ intersecting $E$ 
  properly and write $b_i=b_{E_i}$, $\nu_i=\nu_{E_i}$.
  Assume $E\ne L_\infty$ for simplicity. 
  Then $0=(\cL\cdot E)=b_E(E\cdot E)+\sum_i b_i$
  and $(Z\cdot E)=b_Eg(\nu_E)(E\cdot E)+\sum_i b_ig(\nu_i)$.
  Subtracting, and rearranging 
  using $|\a(\nu_i)-\a(\nu_E)|=(b_ib_E)^{-1}$ yields
  $(Z\cdot E)=b_E\sum_i(g(\nu_i)-g(\nu_E))/|\a(\nu_i)-\a(\nu_E)|$,
  which equals $b_E\Delta g\{\nu_E\}$.
  A similar computation works in the case $E=L_\infty$,
  and completes the proof of the relation  $(Z_\nu\cdot W)=\nu(W)$.
  This relation and the definition of $Z_\nu$ and $Z_\mu$
  imply~\eqref{e:wedge}.

  For the last statement it suffices to
  observe that by the non-degeneracy and unimodularity of the 
  intersection form on $\Pic(X)$, there exists a 
  unique integral class $Z$ satisfying 
  $(Z\cdot W)=\ord_E(W)$ for any $W\in \Pic(X)$.
\end{proof}
\begin{lem}\label{lem:nefcrit}
  If $\nu\in\cV_0$
  then $Z_\nu$ is nef iff $\a(\nu)\ge0$.
\end{lem}
\begin{proof}
  If $Z_\nu$ is nef, then~\eqref{e:wedge} shows that
  $\a(\nu)=(Z_\nu\cdot Z_\nu)\ge0$.
  Conversely, if $\a(\nu)\ge0$, then 
  the definition of $Z_\nu$ shows that $Z_\nu\ge0$ as a 
  function on $\hcVdiv$ and that $\nu\mapsto Z_\nu$ is
  decreasing. The nef cone in $W(\fX)$ being closed,
  it suffices by Lemma~\ref{lem:embedding} to consider 
  the case when $\nu$ is divisorial.

  Hence assume $\nu=b_E^{-1}\ord_E$ is divisorial, where 
  $E$ is a prime in some admissible compactification $X$.
  We must show that $(Z_\nu\cdot C)\ge0$ for every irreducible
  curve $C$ in $X$. If $C$ is a prime of $X$, then this is clear
  by the above. If instead $C$ is the closure of a curve 
  $\{P=0\}$ in $\C^2$, then 
  $(Z_\nu\cdot C)=-\nu(P)$
  which is nonnegative since $\a(\nu)\ge0$.
\end{proof}
\begin{rmk}
  One can also show that $Z_\nu\in\Ltwo(\fX)$ iff 
  $\a(\nu)>-\infty$.
\end{rmk}
% 
%%%%%%%%%%%%%%%%%%%%%%%%%%%%%%%%%%%%%%%%%%%%%%%%%%%%%%%%%%%%%%%%%%%
%
\subsection{Functoriality}
Let $F:\C^2\to\C^2$ be a dominant polynomial mapping. 
Following~\cite{deggrowth} we define actions of $F$ by
pushforward and pullback on classes on the Riemann-Zariski 
space $\fX$. 
The two key facts are 1) given any admissible compactification
$X'$, there exists another admissible 
compactification $X$ such that the lift $X\to X'$ of
$F$ is holomorphic, and 2) given any admissible
compactification $X$ and any prime $E$ of $X$, either $F$
maps $E$ onto a point or a curve in $\C^2$, or there
exists another compactification $X'$ and a prime 
$E'$ of $X'$ such that the lift 
$X\dashrightarrow X'$ of $F$ maps $E$ onto $E'$.

To begin with, we have natural actions 
$F^*:C(\fX)\to C(\fX)$ and $F_*:W(\fX)\to W(\fX)$.
For example, if $\b\in W(\fX)$ is a Weil class, the
incarnation of $F_*\b\in W(\fX)$ on a given 
admissible compactification $X'$ is the
push-forward of $\b_X\in\hodge(X)$ by the 
map $X\to X'$ induced by $F$ for any $X$ such that this
is holomorphic. 

As in~\cite{deggrowth}, the pushforward 
(resp.\ pullback) preserves 
(resp.\ extends to) $\Ltwo$-classes. We obtain bounded 
operators $F_*,F^*:\Ltwo(\fX)\to\Ltwo(\fX)$
and $(F_*\b\cdot\gamma)=(\b\cdot F^*\gamma)$ for
$\b,\gamma\in\Ltwo(\fX)$. 
These operators preserve nef classes. 
We have $F_*F^*=\la_2\cdot\id$ on $\Ltwo(\fX)$ where
$\la_2$ is the topological degree of $F$.
\begin{rmk}
  Note that $F$ defines a dominant rational selfmap of
  $\P^2$, so the constructions in~\cite{deggrowth} apply, 
  but they may
  not yield the same result as above. For example, consider
  the map $F(x,y)=(x^2,xy)$, which contracts the line $x=0$
  to the origin. If $\cL$ is the class of a line in
  $\P^2$, then $F_*\cL=2\cL$ with $F_*$ as above, whereas
  the image of $\cL$ under the pushforward considered 
  in~\cite{deggrowth} is a Cartier class determined 
  in the blowup of $\P^2$ at the origin.
\end{rmk}

\begin{lem}\label{lem:pushpush}
  We have $F_*Z_\nu=Z_{F_*\nu}$ for any $\nu\in\hcV_0$.
\end{lem}
\begin{proof}
  By continuity and homogeneity it suffices to prove this when 
  $\nu=\ord_E$ is divisorial, associated to a prime $E$ of
  some admissible compactification
  $X$ of $\C^2$. 
  Then we can find another admissible 
  compactification $X'$ 
  such that the lift $\tF:X\dashrightarrow X'$ is holomorphic
  and such that $\tF(E)$ is either 
  1) a prime $E'$ of $X'$ or
  2) a point or a curve in $\C^2$.

  In the first case, $F_*\nu=k\ord_{E'}$, where 
  $k$ is the coefficient of $\tF^*E'$ in $E$.
  For any Cartier class $W\in C(\fX)$ we then have
  $(F_*Z_\nu\cdot W)=(Z_\nu\cdot F^*W)
  =\ord_E(F^*W)=k\ord_{E'}(W)=(Z_{F_*\nu}\cdot W)$.

  In the second case, we have $F_*\nu=0$ by definition
  and the same computation above shows that 
  $(F_*Z_\nu\cdot W)=0$ for all $W\in C(\fX)$.
  Indeed, $W$ can be represented as a
  divisor supported at infinity. Hence $F_*Z_\nu=0$.
\end{proof}
The situation for the pull-back is more delicate. For simplicity we state 
a result only in the proper case.
\begin{prop}\label{prop:pull}
  Suppose $F$ is a \emph{proper} polynomial map. 
  Then any $\nu\in\cV_0$ admits at most $\la_2$ preimages
  in $\cV_0$ under $F_\bullet$, and
  one can write:
  \begin{equation}\label{e:pullback}
    F^*Z_\nu = \sum_{F_\bullet \mu = \nu}  a(\mu) Z_\mu,
  \end{equation}
  for some positive constants $a(\mu)$.
\end{prop}
\begin{proof}
  Let $K$ be the function field of $\C^2$.
  The field extension $[K:F^*K]$ has degree $\la_2$. 
  By standard valuation theory, see~\cite{ZS}, for any valuation 
  $\nu$ on $K$ there exists at most $\la_2$ valuations 
  $\mu$ such that $F_*\mu=\nu$. This implies the first assertion.
  
  As $F$ is proper, $F_\bullet$ is an open surjective map 
  for the weak topology on $\cV_0$.
  Note also that if~\eqref{e:pullback} is true, 
  then one has the uniform bound
  $\sum_{F_\bullet^{-1}\{\nu\}} a(\mu)  = 
  (Z_\nu\cdot F_*\cL)\le \deg(F)\a(F_\bullet( -\deg))\le\deg F$. 
  Hence we can reduce the proof to the divisorial
  case by continuity. The proof is then of the same flavor as 
  the proof of Lemma~\ref{lem:pushpush} and left to the reader.
\end{proof}
% 
%%%%%%%%%%%%%%%%%%%%%%%%%%%%%%%%%%%%%%%%%%%%%%%%%%%%%%%%%%%%%%%%%%%
%
\subsection{Dynamics}
Let $F:\C^2\to\C^2$ be a dominant polynomial mapping.
Denote its topological degree by $\la_2\ge1$ 
and set $\la_1\=\lim_{n\to\infty}(\deg F^n)^{1/n}$.
The techniques of~\cite{deggrowth} can be easily adapted
to prove the following result, which corresponds to
parts of Theorem~3.2 and Theorem~3.5 in that paper.
\begin{thm}\label{thm:spectral}
  Assume $\la_2<\la_1^2$. 
  Then there exist nonzero nef Weil classes
  $\theta_*, \theta^*\in \Ltwo(\fX)$ 
  such that $F^*\theta^*=\la_1\theta^*$
  and $F_*\theta_*=\la_1\theta_*$.
  They are unique up to scaling and we may
  normalize them by 
  $(\theta_*\cdot\cL)=(\theta_*\cdot\theta^*)=1$.
  Then $(\theta^*\cdot\theta^*)=0$, 
  $F_*\theta^*=(\la_2/\la_1)\theta^*$ and
  for any Weil class $\theta\in\Ltwo(\fX)$, we have 
  \begin{equation*}
    \frac1{\la_1^n}F^{n*}\theta
    \to(\theta\cdot\theta_*)\theta^*
    \quad\text{and}\quad
    \frac1{\la_1^n}F^n_*\theta
    \to(\theta\cdot\theta^*)\theta_*
    \quad\text{as $n\to\infty$}.
  \end{equation*}
\end{thm}
%
%
%%%%%%%%%%%%%%%%%%%%%%%%%%%%%%%%%%%%%%%%%%%%%%%%%%%%%%%%%%%%%%%%%%%
%
%

\end{document}